\documentclass[12pt,oneside,reqno]{amsart}

\usepackage[a4paper,left=30mm,top=35mm,right=30mm,bottom=35mm]{geometry}

\usepackage[utf8]{inputenc}
\usepackage{amssymb,amsfonts}
\usepackage[all,arc]{xy}
\usepackage{enumerate}
\usepackage{siunitx}
\usepackage{mathtools}
\usepackage{mathrsfs}
\usepackage{dsfont}
\usepackage{amsfonts}
\usepackage{amssymb}
\usepackage{graphicx}
\usepackage{ragged2e}
\usepackage{amsthm}
\usepackage{amsmath,slashed}
\usepackage[mathscr]{euscript}
\usepackage[scr]{rsfso}
\usepackage{graphicx}
\usepackage{color}
\usepackage{float}
\usepackage{caption}
\usepackage{bbm}
\usepackage{subcaption} 
\usepackage{cancel}
\usepackage[shortlabels]{enumitem}

\usepackage[pdftex]{hyperref}
\hypersetup{
	colorlinks=true,
	linkcolor=blue,
	citecolor=blue,
	filecolor=blue,
	urlcolor=blue,
}

\DeclareMathOperator{\vol}{vol}
\DeclareMathOperator{\Var}{Var}

\newtheorem{thm}{Theorem}[section]
\newtheorem{cor}[thm]{Corollary}
\newtheorem{prop}[thm]{Proposition}
\newtheorem{lem}[thm]{Lemma}

\newtheorem{claim}[thm]{Claim}

\theoremstyle{definition}

\newtheorem{exmp}[thm]{Example}

\newtheorem{rem}[thm]{Remark}

\theoremstyle{remark}

\numberwithin{equation}{section}

\def\S{\mathbb{S}}

\def\R{\mathbb{R}}
\def\bZ{\mathbb{Z}}
\def\bP{\mathbb{P}}
\def\bE{\mathbb{E}}

\def\cK{\mathcal{K}}
\def\cH{\mathcal{H}}
\def\cV{\mathcal{V}}
\def\cT{\mathcal{T}}

\def\cN{\mathcal{N}}
\newcommand{\snabla}{\slashed{\nabla}}

\newcommand{\norm}[1]{\left\lVert#1\right\rVert}
\newcommand{\norml}[1]{\|#1\|}
 
\newcommand{\braket}[2]{\langle #1,#2\rangle}
\newcommand{\abs}[1]{\left|#1\right|}
\makeatletter
\newsavebox{\@brx}
\newcommand{\llangle}[1][]{\savebox{\@brx}{\(\m@th{#1\langle}\)}%
	\mathopen{\copy\@brx\kern-0.5\wd\@brx\usebox{\@brx}}}
\newcommand{\rrangle}[1][]{\savebox{\@brx}{\(\m@th{#1\rangle}\)}%
	\mathclose{\copy\@brx\kern-0.5\wd\@brx\usebox{\@brx}}}
\makeatother

\newlength{\intwidth}

\DeclareRobustCommand{\Bint}
{\mathop{%
		\text{%
			\settowidth{\intwidth}{$\int$}%
			\makebox[0pt][l]{\makebox[\intwidth]{$-$}}%
			$\int$}}}

\author{\'Alvaro Romaniega}
\address{Instituto de Ciencias Matem\'aticas, Consejo Superior de
	Investigaciones Cient\'\i ficas, 28049 Madrid, Spain}
\email{alvaro.romaniega@icmat.es}

\author{Andrea Sartori}
\address{ Departement of Mathematics, King's College London, Strand, London WC2R 2LS, England, Uk }
\email{andrea.sartori.16@ucl.ac.uk}

\begin{document}
			\title[Monochromatic waves satisfying the Random Wave Model]{ Nodal set of monochromatic waves satisfying the Random Wave model}
	\begin{abstract}
	We construct deterministic solutions to the Helmholtz equation in $\R^m$ which behave accordingly to the Random Wave Model. We then find the number of their nodal domains, their nodal volume and the topologies and nesting trees of their nodal set  in growing balls around the origin. The proof of the pseudo-random behaviour of the functions under consideration hinges  on a de-randomisation technique pioneered by Bourgain and proceeds via computing their $L^p$-norms. The study of their nodal set relies on its stability properties and on the evaluation of their doubling index, in an average sense. 
	\end{abstract}

\maketitle
\section{Introduction}

\subsection{The Random Wave Model and the nodal set}
\label{RWM}
 Given a compact Riemannian manifold $(M,g)$ without boundary of dimension $m$, let $\Delta_g$ be the Laplace-Beltrami operator. There exists an orthonormal basis for $L^2(M,d\vol)$ consisting of eigenfunctions $\{f_{\lambda_i}\}_{i=1}^\infty$	
\begin{align} \label{Helmholtz equation}
\Delta_g f_{\lambda_i}+ \lambda_i f_{\lambda_i}=0 
\end{align}
with $0=\lambda_1<\lambda_2\leq...$ listed taking into account multiplicity and $\lambda_i\rightarrow \infty$. Quantum chaos is concerned with the behaviour of $f_{\lambda}$ in the \textit{high-energy} limit, i.e. $\lambda \rightarrow \infty$. 

Berry \cite{B1,B2} conjectured that ``generic''  Laplace eigenfunctions on negatively curved manifolds can be modelled in the high-energy limit by monochromatic waves, that is, an isotropic  Gaussian field $F$ with covariance function
\begin{equation}\label{eq:kernel intro}
	\mathbb{E}[F(x)\overline{F(y)}]=\int_{\S^{m-1}}e^{2\pi i\langle x-y, \lambda \rangle} d\sigma(\lambda)= C_m\frac{ J_{\Lambda}\left(|x-y|\right)}{|x-y|^{\Lambda}},
\end{equation}
where $\sigma$ is the uniform measure on the $m-1$-dimensional sphere, $J_{\Lambda}(\cdot)$ is the $\Lambda$-th Bessel function with $\Lambda\coloneqq(m-2)/2$ and $C_m>0$ is some constant such that $\bE[|F(x)|^2]=1$. This is known as the Random Wave Model (RWM) and it is supported by  a large amount of numerical evidence, \cite{HR1992}. 

Noticeably, the RWM provides a general framework to heuristically describe the zero set or  \textit{nodal set} of  Laplace eigenfunctions. In particular, it provides insight into the number of their nodal domains, the connected components of $M\backslash f^{-1}_{\lambda}(0)$, and the volume of their nodal set and  also on its topology. The latter is a key factor in many physical properties, e.g., in Newtonian gravitation, Maxwell electromagnetic theory or Quantum Mechanics \cite{EPS18}.

More precisely, let us denote by $\mathcal{N}(f_\lambda)$ the number of connected components of $M\backslash f^{-1}_{\lambda}(0)$ and by 	$\mathcal{V}(f_{\lambda})\coloneqq \mathcal{H}^{m-1}(\{ x \in M: f_{\lambda}(x)=0\})$ the nodal volume of $f$, where $\mathcal{H}^{m-1}(\cdot)$ is the Hausdorff measure. Then the RWM together with the breakthrough work of Nazarov and Sodin \cite{NS}, suggests that \textquotedblleft typically\textquotedblright, under some conditions,
\begin{align}	\mathcal{N}(f_{\lambda})=c_{NS} \lambda^{m/2}(1+o_{\lambda \rightarrow \infty}(1)) \label{Bour},
\end{align}
where $c_{NS}$ is known as the Nazarov and Sodin constant. Similarly, the RWM together with the Kac-Rice formula  suggests that  \textquotedblleft typically\textquotedblright, under some conditions,
\begin{align}
	\label{volume growth}	\mathcal{V}(f_{\lambda})=c \lambda^{1/2}(1+o_{\lambda \rightarrow \infty}(1)).
	\end{align} 
Importantly, \eqref{volume growth}  agrees with Yau's conjecture \cite{Y82}, which predicts $\mathcal{V}(f_{\lambda}) \asymp \lambda^{1/2}$. The said conjecture is known for real-analytic manifolds thanks to the work of Donnelly-Fefferman \cite{DF}. In the smooth case, the lower bound was recently proved by Logunov and Malinnikova \cite{L2,L1,LM} together with a polynomial upper bound. 
  
We are only aware of one instance when the RWM can be deterministically implemented to obtain information about the nodal set: Bourgain \cite{BU} showed that certain eigenfunctions on the flat two dimensional torus behave accordingly to the RWM and deduced \eqref{Bour}. Subsequently, Buckley and Wigman \cite{BW} extended Bourgain's work to ``generic'' toral eigenfunctions and the second author \cite{SA} proved a small scales version of \eqref{Bour}.
 
 Here, we construct deterministic solutions to \eqref{Helmholtz equation} on  $\mathbb{R}^m$ which satisfy the RWM, in the sense of Bourgain \cite{BU}, in growing balls around the origin. We then use the RWM to study their nodal set, deduce the analogue of \eqref{Bour}, \eqref{volume growth} and also find the asymptotic number of nodal domains belonging to a fixed topological class and with a nesting tree configuration.   These results appear to be new for $m>2$ (the study of the nodal volume also for $m=2$) and they present new difficulties such as the existence of long and narrow nodal domains and the possible concentration of the nodal set in small portions of space. We overcome the far from trivial difficulties using precise bounds on the average \textit{doubling index}, an estimate of the growth rate introduced by Donnelly-Fefferman \cite{DF} (see Section \ref{sec:doubling index}), using recent ideas of Chanillo, Logunov, Malinnikova and Mangoubi, \cite{CLMM20}. In particular, our proofs show how integrability properties of the  doubling index allow to extrapolate information about the zero set  of Laplace eigenfunctions from the RWM. Furthermore, our new approach (based on the weak convergence of probability measures on $C^s$ spaces, Section \ref{SectWeakConv}, and Thom's Isotopy Theorem \ref{ThThom}) gives us an answer to previous questions raised by Wigman and Kulberg, see Section \ref{sec:KW}.
 
\subsection{The eigenfunctions}
\label{stament main reuslts}
Let $m\geq 2$ be a positive integer, $\S^{m-1}\subset \mathbb{R}^m$ be unit sphere and $\{r_n\}_{n\ge 1}\subset \S^{m-1}$ be a sequence  of vectors linearly independent over $\mathbb{Q}$ such that they are not all contained in a hyperplane\footnote{As it will be discusses later, this is a technical, but necessary requirement for our construction to be non-degenerate.}, we will give some properties and examples of such sequences in Section \ref{r_n} below. The functions we study are
\begin{align}
	f_{N}\equiv	f\coloneqq  \frac{1}{\sqrt{2N}}\sum_{|n|\leq N} a_ne(\langle r_n,  \cdot\rangle) \label{function} 
\end{align}
with domain  ${\R}^m$, $a_n$ are complex numbers such that $|a_n|=1$, $e(\cdot)\coloneqq e^{2\pi i\cdot}$ and $\langle \cdot, \cdot \rangle$ is the inner product in $\mathbb{R}^m$. Moreover, we require $\overline{a}_n=a_{-n}$ so that $f$ is real valued, as $r_{-n}\coloneqq-r_n$ for $n>0$.

Differentiating term by term, we see that 
$$ \Delta f=-4\pi^2 f ,$$  
thus, $f$ is a solution of the Helmholtz equation in $\mathbb{R}^m$. Moreover, the \textit{high-energy} limit of $f$ is equivalent to its behaviour in $B(R)=B(R,0)$, the ball of radius $R$ centred at the origin, as $R\rightarrow \infty$. Indeed, rescaling $f$ to $f_R\coloneqq f(R\cdot)$, then 
$$ \Delta f_R= -4\pi ^2 R^2 f_R.$$ 
Thus $(2\pi R)^2$ plays precisely the role of $\lambda$ of Section \ref{RWM}. 

The functions in \eqref{function} do not satisfy any boundary condition, so the spectrum is continuous; however, following Berry \cite{B2}, they can be adapted to satisfy either Dirichlet or Neumann boundary conditions on a straight line. It is plausible that our arguments work also in the boundary-adapted case with minor adjustments, but we do not pursue this here.  Moreover, we have assumed for the sake of simplicity that $|a_{n}|=1$, but more general coefficients could be considered.

Finally, it will be important to keep track of the position of the set $r:=\{r_n\}_{n\ge 1}$ through the following probability measure supported on $\S^{m-1}$:
\begin{align}
\mu_{r,N}=\mu_r= \frac{1}{2N}\sum_{|n|\leq N} \delta_{r_n} \label{muf},
\end{align}
where $\delta_{r_n}$ is the Dirac distribution supported at $r_n$. Since the set of probability measures on $\S^{m-1}$ equipped with the weak$^{*}$ topology is compact (as a standard diagonal argument shows), up to passing to a subsequence, from now on we assume that $\mu_r$ converges to some probability measure $\mu$ as $N\rightarrow \infty$. 

\subsection{Statement of main results, the nodal set of $f$}
\label{main result}

Let $R>1$ and denote by $\mathcal{N}(f,R)$ the number of nodal domains of $f$ in the ball of radius $R$ centred at $0$ which do not intersect $\partial B(R)$, the boundary of $B(R)$, and let $\mathcal{V}(f,R)\coloneqq\mathcal{H}^{m-1}\{x\in B(R): f(x)=0\} $.  Moreover, given a probability measure $\mu$ on $\S^{m-1}$, let $c_{NS}(\mu)$ be the Nazarov-Sodin constant, see Section \ref{additional Tools} below.  
Then, for the functions $f$ as in \eqref{function} we  prove the following  asymptotic statements:
\begin{thm}
	\label{thm 3}
	Let $f$ be as in \eqref{function}, then we have 
	\begin{align}
	\lim_{N\to\infty}\limsup_{R\to \infty}	\left|\frac{\mathcal{N}(f,R)}{ \vol B(R)}- c_{NS}(\mu)\right|=0,
	\end{align} 
	where $c_{NS}(\mu)$ is the Nazarov-Sodin constant of the field $F_\mu$. 
\end{thm}
\begin{rem}
	Note that this kind of double limits gives us the deterministic realizations we are looking for. Indeed, the statement is equivalent to: given some $\varepsilon>0$, then there exist some $N_0=N_0(\varepsilon,m)$ such that all $N\geq N_0$ the following holds: there exists some $R_0=R_0(N,\varepsilon,m)$ such that $R>R_0$, we have
	\begin{align}	
		\left|\frac{\mathcal{N}(f,R)}{ \vol B(R)}- c_{NS}(\mu)\right|\le \varepsilon,
	\end{align}
that is, it satisfies the Nazarov-Sodin growth with a constant as close as we want to $c_{NS}(\mu)$. The question of whether we can take the limit of $N$ first will be analyzed in Section \ref{comments}.
\end{rem}

\begin{thm}
	\label{thm 2}
	Let $f$ be as in \eqref{function}, then  we have 
	\begin{align}\label{eq:nodal volume det}
	\lim_{N\to\infty}\limsup_{R\to \infty}	\left|\frac{\mathcal{V}(f,R)}{\vol B(R)}- c(\mu)\right|=0,
	\end{align} 
	for some (explicit) constant $c(\mu)>0$. 
\end{thm}

\begin{rem}
	Note that, rescaling $f_R=f(R\cdot)$, then Theorem \ref{thm 2} gives 
	$$ 		\lim_{N\to\infty}\limsup_{R\to \infty}\left|\frac{\mathcal{V}(f_R,1)}{R}- c_1(\mu)\right|=0,$$
	for some constant $c_1(\mu)>0$, in accordance with \eqref{volume growth} if $(2\pi R)^2=\lambda$.
\end{rem}

One of the main new ingredient in the proof of Theorem \ref{thm 3} is, in the terminology of Nazarov and Sodin \cite{NS}, the semi-local behaviour of the nodal domains count of $f$, that is, we have the following: 
\begin{prop}
	\label{semi-locality m>2}
	Let $f$ be as in \eqref{function}, $R\geq W\geq 1$. Then, we have 
	\begin{align}
		\frac{\mathcal{N}(f,R)}{\vol B(R)}= \frac{1}{\vol B(W)}\Bint_{B(R)}\mathcal{N}(f,B(x,W))dx +	O\left( W^{-1} \right) + O_{N,W}\left(R^{-\Lambda-3/2}\right). \nonumber
	\end{align}
	\end{prop}
For $m=2$, Proposition \ref{semi-locality m>2} follows from the bound $\mathcal{	V}(f,R)\lesssim R^m$, see for example Section \ref{sec:doubling index} below, which implies that most nodal domains have diameter at most $O(1)$. However, for $m>2$, this argument does not rule out the existence of many long and narrow nodal domains. Following  the recent preprint of  Chanillo, Logunov, Malinnikova and Mangoubi \cite{CLMM20}, $f$ should grow fast around such nodal domains and this can be estimated in terms of the \textit{doubling index} of $f$, see Section \ref{sec:doubling index} below. The proof of Proposition \ref{semi-locality m>2} then relies on precise estimates on the average growth of $f$, which we obtain in Section \ref{growth of $f$}. Using the aforementioned estimates,  we are also able to show that there is no concentration of nodal volume of $f$ in small portion of the space. That is, we prove the following proposition which will be one of the main ingredients in the proof of Theorem \ref{thm 2}:
 \begin{prop}
	\label{anti-concentration prop}
	Let $F_x$ be as \textnormal{\eqref{Fx}}, then for some (fixed) $\alpha>0$ there exists $R_0=R_0(N,W,\alpha)$ such that for all $R\geq R_0$, we have 
	\begin{align}
		\Bint_{B(R)}\mathcal{V}(f,B(x,W))^{1+\alpha}dx\lesssim W^{m(1+\alpha)}+W^{(m-1)(1+\alpha)^2}+O_{N,W,\alpha}(R^{-\Lambda-3/2}),\nonumber
	\end{align}
where the constant in the $\lesssim$-notation is independent of $N$. 
	\end{prop}
\subsection{De-randomisation}

In this section we make precise in which sense $f$ satisfies the RWM. We first need to introduce some notation: let $R\gg W>1$ be some parameters, where $R$ is much larger than $W$, and let $F_{x,W,R,N}=F_x$ be the restriction of $f$ to $B(x,W)$, the ball of radius $W$ centred at $x\in B(R)$, that is,
\begin{align}\label{Fx}
F_x(y)\coloneqq \frac{1}{\sqrt{2N}}\sum_{|n| \leq N} a_n e( r_n\cdot  x)e\left(r_n \cdot  y \right)
\end{align}
for $y\in B(W)$ and $x\in B(R)$. Here, we show that, as we sample $x$ uniformly in $B(R)$, the ensemble $\{F_x\}_{x\in B(R)}$  approximates, arbitrarily close, the centred stationary Gaussian field with spectral measure $\mu$. We denoted the said field by $F_{\mu}$ and collect  the relevant background in Section \ref{Gaussian random fields} below. 

 To quantify the distance between $F_x$ and $F_{\mu}$, given some integer $s\geq 0$ and $W\geq 1$, we consider their \textit{pushforward} probability measures (see Section \ref{sec:notation} below) on the space of (probability) measures on $C^s(B(W))$, the class of $s$ continuously differentiable functions on $B(W)$. Since the space of probability measure on $C^s(B(W))$ is metrizable via the Prokhorov metric $d_P$, we define the distance between $F_x$ and $F_{\mu}$ as the distance between their pushforward measures.  More precisely, given to random fields $F,F'$ defined on two, possibly different, probability spaces with measures $\mathbb{P}$ and $\mathbb{P}'$, we write $d_P(F,F')\coloneqq d_P(F_*\mathbb{P},F'_*\mathbb{P}')$, where $F_*\mathbb{P}$  is the pushforward probability measure. We collect the relevant background in Section \ref{SectWeakConv} below.
 
With this notation, we prove the following:
\begin{thm}
	\label{thm 1}
	Let $f$ and $F_x$ be as in \eqref{function} and \eqref{Fx} respectively, $W>1$, and $s\geq 0$. Then we have 
\begin{equation}\nonumber
\lim_{N\to\infty}\limsup_{R\to \infty}	d_{P}(F_x,F_{\mu})=0,
\end{equation}
where the convergence is with respect to the $C^s(B(W))$ topology. 
\end{thm} 

One of the main ingredients in the proof of Theorem \ref{thm 1} will be the computation of the $L^p$-norms of $F_x$ and from these deduce its Gaussian behaviour. In particular, in  Proposition \ref{prop:Lp mom} we will show that 
$$\frac{1}{\vol B(R)}\int_{B(R)} |F_x(y)|^{2p}dx= \frac{(2p)!}{p! 2^p}(1+o_{N,R\rightarrow \infty}(1)),$$
uniformly for $y\in B(W)$, that is, $F_x$ has (asymptotically) real Gaussian moments.
\subsection{Topologies and nesting trees}
\label{topology and trees}
In this section we present a strengthening of Theorem \ref{thm 3} in that we study nodal domains restricted to a particular topological class or nesting tree. First, we need to introduce some definitions following \cite{SW}.  Let $\Sigma\subset\R^m$ be  a smooth, closed, boundaryless, orientable
submanifold and denote by $[\Sigma]$ its diffeomorphism class, that is, $\Sigma'\sim\Sigma$ if and only if there exists a diffeomorphism $\Phi$ such that $\Phi\left(\Sigma\right)=\Sigma'$, and let $H(m-1)$ be the set of diffeomorphism types $[\Sigma]$. Moreover,  since $V(R):=f^{-1}(0) \cap B(R)$ is a smooth $m-1$-dimensional manifold (if the zero set is regular), we can decompose $V(R)$ into its connected components $V(R)= \bigcup_{c\in \mathcal{C}(f;R)} c$, where we ignore components which intersect $\partial B(R)$. Similarly, we can decompose $B(R)\backslash V(R)= \bigcup_{\alpha\in \mathcal{A}(f;R) } \alpha$ as an union of connected components. We define the tree $X(f;R)$ where the vertices are $\alpha \in \mathcal{A}(f,R)$ and there is an edge between $\alpha,\alpha' \in \mathcal{A}(f;R)$ if the share an (unique) common boundary $c\in \mathcal{C}(f;R)$. Let $\cT$ be the set of finite rooted trees.

 We define $\mathcal{N}(f,\mathcal{S},R)$ where $\mathcal{S}=\{[\Sigma_\beta]\}_{\beta\in \mathfrak{B}_\mathcal{S}}\subset H(m-1)$ as the number of nodal components of $f$ in $B(R)$, which do not intersect $\partial B(R)$ and diffeomorphic to some $\Sigma_\beta\in\mathcal{S}$. Given $T\in \cT$, we define $\cN(f, T ,R)$ similarly. With this notation, we prove the following: 
\begin{thm} \label{thm 5}
 Let $f$ be as in \eqref{function},  $\mathcal{S}\subset H(m-1)$ and $T\in \mathcal{T}$, then we have 
	\begin{align}
\lim_{N\to\infty}\limsup_{R\to \infty}	\left|\frac{\mathcal{N}(f,\mathcal{S},R)}{\vol B(R)} - c(\mathcal{S},\mu)\right| =0 \nonumber \\
\lim_{N\to\infty}\limsup_{R\to \infty}\left| \frac{\mathcal{N}(f,T,R)}{\vol B(R)} - c(T,\mu)\right| =0, \nonumber
	\end{align}
for some constants $c(\mathcal{S},\mu)$ and $c(T,\mu)$. 
\end{thm}
We observe that Theorem \ref{thm 3} follows from Theorem \ref{thm 5} choosing $\mathcal{S}= H(n-1)$. Therefore, we only need to prove Theorem \ref{thm 5}. 

\subsection{Examples and properties of the $r_n$'s}
\label{r_n}
In this section, we give two examples of sequences $\{r_n\}\subset \S^{m-1}$ being $\mathbb{Q}$-linearly independent. 
\begin{exmp}
	For $m=2$, identifying $\S^1$ with $\R/\bZ\simeq[0,1]$, we may	take a sequence of rational numbers $\{b_n\}$ in $(1,e)$ then $\log b_n=r_n$ is linearly independent over $\mathbb{Q}$ by Baker's theorem \cite{B}. For $m>2$, we may take a vector the first co-ordinate of which is $\log b_n$. 
\end{exmp}
\begin{exmp}
	\label{example2}For $\S^{m-1}$, we can construct the sequence as follows. Let $r_1$ be a point on $\S^{m-1}$ and define $S_1\coloneqq \S^{m-1}\backslash \overline{\mathbb{Q}}r_1$, the span with algebraic coefficients of $r_1$. As we are removing a countable set from an uncountable set, $S_1$ is non-empty, in fact, uncountable, thus we may choose any $r_2\in S_1$. For a general $n\geq 2$, let $S_n\coloneqq S_{n-1}\backslash \overline{\mathbb{Q}}r_{n-1} $	and $r_{n}\in S_n$. By induction, bearing in mind that for sets $A,B,C$, $(A\backslash B)\backslash C= A\backslash (B\cup C)$, the sequence is rational independent and, by construction, we can also choose the $r_n$'s such that they uniformly distribute over $\S^{m-1}$. In particular, we may choose a sequence of $r_n$ such that $\mu_r$ weak$^{*}$ converges to the Lebesgue measure on $\S^{m-1}$. 
\end{exmp}
In particular, Example \ref{example2} shows that if we chose the $r_n$ uniformly at random from $\S^{m-1}$, then the rational independence assumption would hold almost surely. This implies that our assumptions are somehow ``generic''. Finally, we will repeatedly use the following consequence of the $\mathbb{Q}$-linear independence of the vectors $\{r_n\}$: by a compactness argument, for any $N>1$ and $T\geq 1$ there exists some $\gamma=\gamma(N,T)$ such that for any $t\leq T$ 
\begin{align}
\left|	r_{n_1}+...+r_{n_t} \right| >\gamma(N,T)>0, \label{1}
\end{align}
for all $|n_1|\leq N$,...,$|n_t|\leq N$, unless $t$ is even and, up to permuting the indices,  $r_{n_1}= -r_{n_2},...,r_{n_{t-1}}=-r_{n_t}$.

\subsection{Plan of the proofs}
\label{plan fo the proofs}
\textit{Proof of Theorem \textnormal{\ref{thm 1}}, Section \textnormal{\ref{proof thm 1}}}. The proof of Theorem \ref{thm 1} follows from an application of Bourgain's de-randomisation:  roughly speaking, the linear independence of the sequence $\{r_n\}$ implies \textit{asymptotic} independence of the waves $e(\langle r_n,x\rangle)$ under the uniform measure in $B(R)$, thus the asymptotic Gaussian behaviour of $F_x$ as in \eqref{Fx} is \textit{expected} from the Central Limit Theorem, although we cannot directly apply the CLT as our waves are not independent.

To make this intuition precise, following Bourgain, we introduce an additional parameter $K\geq 1$ and consider an auxiliary function:
\begin{equation}\label{eq:def phi}
	\phi_x(y)\coloneqq \sum_{k\in\cK} \left[\frac{1}{\left(2N\mu_r(I_k)\right)^{1/2}}\sum_{r_n\in I_k} a_n e(r_n\cdot x)\right]{\mu_r(I_k)^{1/2}}e\left( \zeta^k\cdot y\right) 
\end{equation}
where the $\zeta^k \in I_k\subset \S^{m-1}$ for $k\in \cK$ are appropriately chosen points and the $I_k$ form a particular subset of a partition of the sphere, see \eqref{phi}. First, in Lemma  \ref{first approx}, using asymptotic results for Bessel functions, we show that  $\phi_x$ is, on average, a good approximation of $F_x$ as the number of $\zeta^k$ grows, that is,
\begin{align}
\nonumber
\Bint_{B(R)} \norm{\phi_x-F_x}_{C^s(B(W))}^2dx =o(1)~~~\text{ as }K,R \rightarrow \infty.
\end{align}
 The advance in passing to $\phi_x$ is that we isolate the contribution of the ``wave-packets'' 
$$
b_k\coloneqq\frac{1}{(2N\mu_r(I_k))^{-1/2}}\sum_{r_n\in I_k} a_n e(r_n\cdot R x);
$$
this allows us to show, see Lemma \ref{moments}, that the $b_{k}'s$ are asymptotically (as $N, R\rightarrow \infty$) i.i.d. complex standard Gaussian random variables. Thus, we can ``approximate'' $\phi_x$ in the $C^s(B(W))$ topology by the random field
\begin{equation}\label{eq:def muK}
\kappa_K F_{\mu_K}(y) \coloneqq \sum_{k\in \cK}\mu_r(I_k) c_ke\left( \zeta^k , y\right) 
\end{equation}
with the $c_{k}$ i.i.d. complex standard Gaussian random variables and $\kappa_K$ a normalizing factor, see \eqref{DefmuK}. Finally, we let $K$ go to infinity so that the field $F_{\mu_K}$ will ``converge'' to $F_{\mu}$. We observe that passing to $\phi_x$ gives a stronger statement than Theorem \ref{thm 1} because $\phi_x$ and $F_x$ are defined on the same probability space and are $C^s$ close in $L^2$, not just with respect to the Prokhorov distance. \\

\textit{Proof of Theorem \textnormal{\ref{thm 5}}, Sections \textnormal{\ref{last section}} and \textnormal{\ref{section 5}}}. We discuss the proof of the (simpler) Theorem \ref{thm 3}. The  starting point is  Proposition \ref{semi-locality m>2}: 
\begin{align}\frac{\mathcal{N}(f,R)}{\vol B(R)}= \frac{1}{\vol B(W)}\Bint_{B(R)} \mathcal{N}(F_x,W)dx +O\left(\frac{1}{W}\right)+ O_{N,W}\left(\frac{1}{R^{(m+1)/2}}\right). \label{semi-locality formula}
\end{align}
As mentioned in the introduction, to prove \eqref{semi-locality formula}, we need to discard the possibility of long and narrow nodal components of $f$ which intersect many balls $B(x,W)$. Following the recent preprint of  Chanillo, Logunov, Malinnikova and Mangoubi \cite{CLMM20},  $f$ has to grow very fast in balls around such nodal domains, this can be quantified using the \textit{doubling index}\footnote{In the literature, the doubling index is usually denote by $N$ or $\cN$. Since this would clash with the $N$ in \eqref{function} or the $\cN$ of nodal domains, we opted for $\mathfrak{N}(\cdot)$. We will slightly modify the definition later. } of $f$ in a ball $B(x,W)$: 
$$ \mathfrak{N}_f(B(x,W)):= \log \frac{\sup_{B(x,2W)}|f|}{\sup_{ B(x,W)}|f|} +1. $$
In Lemma \ref{anti-concentration v2}, we show that $\mathfrak{N}_f(x,W)$ is not too big in an appropriate average sense. Therefore long and narrow nodal domains are ``rare'' and contribute only to the error term in  \eqref{semi-locality formula}. This will be the content of Section \ref{last section}.

Next, we show that Theorem \ref{thm 1} together with the stability of the nodal set (Proposition \ref{PropContBdd}) imply that 
\begin{align}
	\mathcal{N}(F_x,W) \overset{d}{\longrightarrow}\mathcal{N}(F_{\mu},W)\quad\text{ as }   N,R\rightarrow \infty, \label{1.6.1}
\end{align}
where the convergence is in distribution. Thanks to the Faber-Krahn inequality \cite[Chapter 4]{Cbook}, see also \cite[Theorem 1.5]{M}, $$\sup_x\mathcal{N}(F_x,W)\lesssim W^m,$$ thus, uniform integrability or Portmanteau Theorem, together with  \eqref{semi-locality formula} and \eqref{1.6.1} give
\begin{align}\Bint_{B(R)} \mathcal{N}(F_x,W)dx = \mathbb{E}[ \mathcal{N}(F_{\mu},W)](1+o(1))\quad\text{ as }N,R\rightarrow \infty. \label{1.1.2}
\end{align}
This is proved in Proposition \ref{PropConvE2}. Finally, we evaluate the right hand side of \eqref{1.1.2} using the work of Nazarov-Sodin \cite{NS}, thus concluding the proof of Theorem \ref{thm 3}.

\textit{Proof of Theorem \textnormal{\ref{thm 2}}, Section \textnormal{\ref{nodal volume}}}. The proof of Theorem \ref{thm 2} follows the same strategy as the proof of Theorem \ref{thm 3}, with the additional difficulty that $\mathcal{V}(F_x)$ may be unbounded in the supremum norm. To circumvent this problem, and thus apply the uniform integrability theorem, we show  in Proposition \ref{anti-concentration prop} that $ \mathcal{V}(F_x,W)$  is uniformly integrable.  The proof relies on the estimate on $\mathfrak{N}_f(x,W)$ which we obtained in Section \ref{growth of $f$}. Once Proposition \ref{anti-concentration prop} is proved, the proof of Theorem \ref{thm 2} follows step by step the proof of Theorem \ref{thm 3}.  

Finally in Section \ref{comments} we collect some final comments and in the appendix some proofs for completeness.

\subsection{Related work}\label{SectRelatedWork}
\textit{De-randomisation}. Ingremeau and  Rivera \cite{IR20} applied the technique on Lagrangian states, that is, functions of the form $f_h(x)=a(x)e^{i\theta(x)/h}$. The authors show that the long time evolution by the semiclassical Schr\"{o}dinger operator of (a wide family of) Lagrangian states on a negatively curved compact manifold satisfies the RWM in a sense similar to Theorem \ref{thm 1}. Thus, they provide a family of functions on negatively curved manifolds satisfying the RWM. 

 \textit{Nodal domains}. The study of $\mathcal{N}$ for Gaussian fields started with the breakthrough work of Nazarov and Sodin \cite{NS09,NS}. They found the asymptotic law of the expected number for nodal domains of a stationary Gaussian field in growing balls, provided its spectral measure satisfies certain (simple) properties, importantly the spectral measure should not have atoms. That is, given a (nice) Gaussian field with spectral measure $\mu$,  there exists some constant $c_{NS}(\mu)>0$ such that 
 \begin{align}\lim\limits_{R \rightarrow \infty} \frac{\mathcal{N}(F_{\mu},R)}{\vol(B(R))}= c_{NS}(\mu), \label{NSresult}
 	\end{align}
 where the convergence is a.s. and in $L^1$. 
  
 As far as deterministic results about $\mathcal{N}$ are concerned,   Ghosh, Reznikov and Sarnak \cite{GRS1,GRS2}, assuming the appropriate Lindel\"{o}f  hypothesis, showed that $\mathcal{N}(\cdot)$ grows at least like a power of the eigenvalue for individual Hecke-Maass eigenfunctions. Jang and Jung \cite{JJ} obtained unconditional results for individual Hecke-Maass eigenfunctions of arithmetic triangle groups. Jung and Zelditch \cite{JZ}  proved, generalising the geometric argument in \cite{GRS1,GRS2}, that  $\mathcal{N}(\cdot)$ tends to infinity, for most eigenfunctions on certain negatively curved manifolds, and Zelditch \cite{Z2} gave a logarithmic lower bound. Finally,   Ingremeau \cite{I} gave examples of eigenfunctions with $\mathcal{N}(\cdot) \rightarrow \infty$ on unbounded negatively-curved manifolds. 

\textit{Topological classes}.  Sarnak and Wigman \cite{SW} and Sarnak and Canzani \cite{CS19} proved the analogous result of \eqref{NSresult} for $\mathcal{N}(F_{\mu}, T,R)$ and $\mathcal{N}(F_{\mu}, H,R)$, again, for spectral measures with no atoms. For deterministic results, Enciso and Peralta-Salas \cite{EPS13} proved the existence of functions $g$ (in the more general setting of elliptic equations and non-necessarily compact components) such that $\mathcal{N}(g,H,R)>0$ and this property is valid even if we perturb $g$ in a $C^k$ norm. This is the key element to prove the positivity of the constants $c(H,\mu)$ of the analogous result of \eqref{NSresult}. It is also worth mentioning that Enciso and Peralta-Salas' techniques can be applied to solve another problem raised by M. Berry \cite{Be01} related to the existence of (complex) eigenfunctions of a quantum system whose nodal set has components with arbitrarily complicated linked and knotted structure, \cite{EHP15}. Furthermore, somehow related techniques for the construction of specific structurally stable examples applied to dynamical systems play a fundamental role in an extension of Nazarov-Sodin's theory to Beltrami fields. These fields are (vector-valued) eigenfunctions of the $\text{curl}$ (instead of the Laplacian treated here) and they are a key element in fluid dynamics; turbulence can only appear in a fluid in equilibrium through Beltrami fields. This extension allows one to stablish V. I. Arnold's long standing conjecture on the complexity of Beltrami fields (i.e., a typical Beltrami field should exhibit chaotic regions
coexisting with a positive measure set of invariant tori of complicated topology), see \cite{EPRBeltrami}.

 \subsection{Notation}
 \label{sec:notation}
We will use the standard notation  $\lesssim$ to denote $\le C$, where the constant can change its value between equations, and $m\geq 2$ will be a positive integer which denotes the dimension of the space and $\Lambda=(m-2)/2$. Moreover, given a large parameter $R>1$, we denote by $B(R)$ the ball of radius $R$ in $\mathbb{R}^m$ and by $\overline{B(R)}$ its closure. Given some $r>0$ and a ball $B$, we denote by $rB$ the concentric ball with $r$-times the radius. We write
\begin{align}
	\Bint_{B(R)} h(x)dx\coloneqq\frac{1}{\vol B(R)}\int_{B(R)} h(x)dx =\int_{B(R)} h(x)d\vol_R(x) \nonumber
\end{align}
where $\vol_R$ for the uniform probability measure on $B(R)$. Furthermore, we denote by $(\Omega,\mathbb{P})$ an abstract probability space where every random object is defined and, given a probability measure $\mu$ on $\S^{m-1}$, we denote by $F_{\mu}$ the centred, stationary Gaussian field with spectral measure $\mu$, see Section \ref{Gaussian random fields} for more details. 

Given two measurable spaces $(Y,\Sigma)$ and $ (X,\mathcal{F})$, a measurable mapping ${ g\colon Y\to X}$ and a measure $\mu$ on $Y$, the \textit{pushforward} of $\mu$, denoted by $g_{*}\mu$, is
$$
{ g_{*}\mu(B)\coloneqq\mu \left(g^{-1}(B)\right)}
$$
for ${ B\in \mathcal{F}}$. Note that  $g_{*}\mu$ is well-defined as $g$ is measurable. Finally, given some function $g: \mathbb{R}^m \rightarrow \mathbb{R}$ and a set $A\subset \mathbb{R}^m$, we denote by $g\rvert_A$ the restriction of $g$ to $A$. 
\section{Preliminaries}
\label{preliminaries}
\subsection{Gaussian fields background}
\label{Gaussian random fields}
We  briefly collect some definitions about Gaussian fields (on $\mathbb{R}^m$). For us, a (real-valued) Gaussian field $F$ is a continuous map $F: \mathbb{R}^m \times \Omega\rightarrow \mathbb{R}$ for some probability space  $\Omega$,  such that all  finite dimensional distributions $(F(x_1, \cdot),...F(x_k,\cdot))$ are multivariate Gaussian. We say that $F$ is \textit{centred} if $\mathbb{E}[F]\equiv0$ and \textit{stationary} if its law is invariant under translations $x\rightarrow x+\tau$ for $\tau \in \mathbb{R}^m$. In this script, every Gaussian field is both centred and stationary. Then, the \textit{covariance} function of $F$ is 
\begin{align}
\mathbb{E}[F(x)\cdot F(y)]= \mathbb{E}[F(x-y)\cdot F(0)]. \nonumber
\end{align}
Since the covariance is positive definite, by Bochner's theorem, it is the Fourier transform of some measure $\mu$ on $\mathbb{R}^m$. So we have 
\begin{align}
\mathbb{E}[F(x)F(y)]= \int_{\mathbb{R}^m} e\left(\langle x-y, \lambda \rangle\right)d\mu(\lambda). \nonumber
\end{align}
The measure $\mu$ is called the \textit{spectral measure} of $F$ and, since $F$ is real-valued, it satisfies ${\mu(-I)}=\mu(I)$ for any (measurable) subset $I\subset \mathbb{R}^m$, that is, $\mu $ is a symmetric measure. By Kolmogorov theorem, $\mu$ fully determines $F$, so we simply write $F=F_{\mu}$.

\subsection{Weak convergence of probability measures in the $C^s$ space.\\}\label{SectWeakConv}
Let $S=C^s(V)$ be the space of $s$-times, $s\ge0$ integer, continuously differentiable functions on $V$, a compact set of $\R^m$. In this section we review the conditions to ensure that a sequence of probability measures $\{\mu_n\}$ on $S$ converges weakly to another probability measure, $\mu$, see also \cite[Chapter 7]{BI} for  $s=0$. 

First, since $S$ is a separable metric space, Prokhorov's Theorem \cite[Chapters 5 and 6]{BI} implies that $\mathcal{P}(S)$, the space of probability measures on $S$, is metrizable via the \textit{L\'evy–Prokhorov metric}. This is defined as follows:  for a subset $B\subset S$, let denoted by $B_{{+\varepsilon }}$ the  $\varepsilon$-neighbourhood of $B$, that is, 
$$
B_{{+\varepsilon }}:=\{p\in S~|~\exists~ q\in B,\ \norm{p-q}_{C^s}<\varepsilon \}.
$$
 The \textit{L\'evy–Prokhorov metric} $d_P :{\mathcal  {P}(S)\times{P}}(S)\to [0,+\infty )$ is defined for two probability measures $\mu$  and $\nu$  as:
\begin{align}\label{def of d_p}
	d_P (\mu ,\nu ):=\inf_{\varepsilon>0} \left\{\mu (B)\leq \nu (B_{{+\varepsilon }})+\varepsilon, \ \nu (B)\leq \mu (B_{{+\varepsilon }})+\varepsilon \ \forall~ B\in S\right\}. 
\end{align}

It is well-known \cite[Claim below Lemma 2]{Pri93} and \cite{Wil86} that if the finite dimensional distributions of some sequence $X_n$ taking values on $S$ converge to some random variable $X$, that is for all $y_1,...,y_l\in V$ 
\begin{align}
	(X_n(y_1),....,X_n(y_l))\overset{d}{\longrightarrow} (X(y_1),....,X(y_l)) \quad \text{ as }n\rightarrow \infty
	\end{align}
where the convergence is in distribution, and the sequence $\{(X_n)_{*}\mathbb{P}\}$ is \textit{tight}, then $(X_n)_{*}\mathbb{P}$ converges to $(X)_{*}\mathbb{P}$ in $\mathcal{P}(S)$ equipped with the metric $d_P$. A set of probability measured $\Pi$ on $S$ is tight if for any $\varepsilon >0$ there exists a compact subset $Q_{\varepsilon } \subset S$ such that, for all measures $\nu \in \Pi$, $\nu (Q_{\varepsilon}) > 1 - \varepsilon.$

A characterization of tightness in  $\mathcal{P}(S)$ is given in the next lemma, which can be seen as a probabilistic version of Arzelà-Ascoli Theorem.  Let us define the modulus of continuity of a function $g\in S$ as: 
\begin{align}
	\label{continuity mod}
\omega_{g}(\delta)\coloneqq \sup_{\norm{y-y'}\le\delta}\{|g(y)-g(y')|\}.
\end{align}
We then have following lemma \cite[Lemma 1]{Pri93}:
\begin{lem}
	\label{tightness} A sequence $\{\mu_n\}$ of probability measures on $S$ is tight if and only if
	\begin{itemize}
		\item[\textnormal{i)}] 	For some $y\in V$ and $\varepsilon>0$ there exists $M>0$ such that, uniformly in $n$:
		\begin{equation}\nonumber
		\max_{|\alpha|\leq s}~\mu_n(g:|D^\alpha g(y)|>M)\le \varepsilon\quad.
		\end{equation}
	
		\item[\textnormal{ii)}] For all multi-index $\alpha$ such that $|\alpha|=s$ and $\varepsilon>0$, we have 
			\begin{equation} \nonumber
\lim_{\delta\to 0}\limsup_{n\to \infty}	\mu_n(g:\omega_{D^\alpha g}(\delta)\ge \varepsilon)=0.
		\end{equation}	
		
	\end{itemize}
\end{lem}
 
Finally, we will need the following result of uniform integrability \cite[Theorem 3.5]{BI}.
\begin{lem}\label{th:DCTh}
	Let $X_n$ a sequence of random variables such that $X_n\overset{d}{\rightarrow} X$ (i.e., in distribution). Suppose that there exists some $\alpha>0$ such that $\mathbb{E}[|X_n|^{1+\alpha}]\le C<\infty$ for some $C>0$, uniformly for all $n\geq 1$. Then,
	\begin{equation}
		\bE X_n\to \bE X. \nonumber
	\end{equation}
\end{lem}
\subsection{Doubling index}\label{sec:doubling index}
Following  and Donnelly-Fefferman \cite{DF} and Logunov and Malinnikova\cite{L2,L1,LM}, given a function $h:\mathbb{R}^m\rightarrow \mathbb{R}$, we define the \textit{doubling index} of $h$ in $B$ as 
\begin{align}
	\label{def doubling index}
	\mathfrak{N}_h(B)\coloneqq \log \frac{\sup_{\varkappa_m B} |h|}{\sup_{B }|h|} +1,
\end{align}
with $\varkappa_m\coloneqq2\sqrt{m}$. The doubling index gives a bound on the nodal volume of $f$, as in \eqref{function}, thanks to the following result \cite[Proposition 6.7]{DF} and \cite[Lemma 2.6.1]{LMlecturenotes}. 
 \begin{lem}
 	\label{doubling index}
 	Let $B \subset \R^{m}$ be the unit ball, suppose that $h: 3B\rightarrow \mathbb{R}$ is an harmonic function, that is, $\Delta h=0$,	then 
 	\begin{align}
 		\mathcal{V}(h, 1/2)\lesssim  \mathfrak{N}_h(B) .   \nonumber
 	\end{align}

 \end{lem}
Applying Lemma \ref{def doubling index} to the lift $h(x,t)\coloneqq f(x)e^{2\pi t}:\mathbb{R}^{m+1}\rightarrow \mathbb{R}$, we obtain the following: 
 \begin{lem}
 	\label{doubling f}
	Let $f$ be as \textnormal{(\ref{function})} and $r>1$ be some parameter, then
	\begin{align}
		\mathcal{V}(f,B(r)) \cdot r^{-m+1} \lesssim \mathfrak{N}_f(B(3r))+r.  \nonumber
	\end{align} 
\end{lem}
\begin{proof}
	First, we observe that the function $h(x,t)\coloneqq f(x)e^{2\pi t}$ is harmonic in a ball  ${B}(\sqrt{2}r) \supset B(r)\times[-r,r]$ and that 
	$$  \mathcal{H}^{m-1}\{x\in B(r): f(x)=0\} \times 2r \leq \mathcal{H}^{m} \{(x,t)\in {B}(\sqrt{2}r): h(x,t)=0\}$$
	Therefore, rescaling ${B}(\sqrt{2}r)$ to a ball of radius one, the lemma follows from Lemma \ref{doubling index}, upon noticing that
	$$
	\cV(f,B(r))r^{-m+1}\lesssim r+ \mathfrak{N}(f,B(cr))
	$$	
	for any $c>2\sqrt{2}$ and that the supremum norm is scale invariant.  
\end{proof}
In particular, we can control the doubling index of $f$ using the well-known Nazarov-Turan Lemma, see \cite{N} and \cite{FM} for the multi-dimensional version:
\begin{lem}
	\label{Nazarov-Turan}
	Let $g(x)= \sum_{j=1}^J a_j e( \xi_j \cdot x)$ for $x\in \mathbb{R}^m$ and $\xi_1,...\xi_J$ distinct frequencies, moreover let $B\subset \R^{m}$ be a ball and $I\subset B $ be a measurable subset. Then there exist absolute constants $c_1,c_2>0$ so that 
	\begin{align}
		\sup_{B}|g|\lesssim \left(c_1\frac{|B|}{|I|}\right)^{c_2J}\sup_{I}|g| \nonumber
	\end{align}
\end{lem}
 Combining Lemma \ref{Nazarov-Turan} with Lemma \ref{doubling f}, we obtain the following:
 \begin{lem}
	\label{Yau length}
	Let $f$ be as \textnormal{(\ref{function})} and $r>0$ be some parameter, then
	\begin{align}
		\mathcal{V}(f,B(r)) \cdot r^{-m+1} \lesssim N+r.  \nonumber
	\end{align} 
\end{lem}

Finally, to study the nodal domains of $f$, we will to use the doubling index to control the growth of $f$ in sets which might not be balls. That is, we will   need the following lemma \cite{LMlecturenotes}: 
\begin{lem}[Remez type inequality]
	\label{Remez type inequality} 
	Let $B$ be the unit ball in $\mathbb{R}^m$ and suppose that $h:2B \rightarrow \mathbb{R}$ ia an harmonic function. Then there exist constants $c_1,c_2>0$, independent of $h$, such that 
	$$ \sup_B |h| \lesssim \sup_E |h|\left( c_1\frac{|B|}{|E|}\right)^{c_2 \mathfrak{N}_h(2B)}$$
	for any set $E\subset B$ of positive measure. 
\end{lem}
 Using the harmonic lift $h$ of $f$ as in Lemma \ref{doubling f} and rescaling, we deduce the following:
 \begin{lem}
	\label{Remez for f} 
	Let $B(r) \subset\mathbb{R}^m$ be a ball of radius $r>0$ and $f$ be as in \eqref{function}	then there exist constants $c_1,c_2>0$, such that 
	$$ \sup_{B(r)} |f| \lesssim \sup_{E(r)} |f|  \left( c_1\frac{|B(r)|}{|E(r)|}\right)^{c_2 (\mathfrak{N}_f(B(2r))+r)}$$
	for any set $E(r)\subset B(r)$ of positive measure. 
\end{lem}
\subsection{Additional Tools}
\label{additional Tools}
In this section we extend for our purposes the work of Nazarov-Sodin \cite{NS} and Sarnak-Wigman \cite{SW} to the case of a  possibly atomic symmetric spectral measure and give a sufficient condition for the positivity of the constants $c_{NS}(\cdot)$, $c(T,\cdot)$ and $c(H,\cdot)$ appearing in Theorems \ref{thm 1} and \ref{thm 5}. For dimension two and for nodal domains, this was done in \cite[Proposition 1.1]{KW2}, see also Section \ref{sec:KW} below for some additional results.  The proof essentially follows \cite{NS}, we reproduce some details for completeness.

Given a probability measure $\mu$ on $\S^{m-1}$ and an integer $s\geq 1$, let $\overline{\mathcal F\, L^2_{\tt H}(\mu)}^{C^s}$, the closure in the Fr\'echet topology of $C^s$ compact convergence of the Fourier transform of Hermitian
functions $h\colon \R^m\to\mathbb C$ with $\int |h|^2\, d\mu < \infty $. Then, bearing in mind the notation in Section \ref{topology and trees}, we have the following:

\begin{thm}
	\label{ThNSSWatomic}
	Let $\mu$ be symmetric probability measure on $\S^{m-1}$. Let $\mathcal{S}\subset H(m-1)$ and $T\in \cT$. Then, there exist constants $c(\mathcal{S},\mu), c(T,\mu)$ such that
	\begin{align}
		\mathbb{E}[ \mathcal{N}(F_\mu,\cdot,R) ]&= \vol B(R)(c(\cdot,\mu) + o_\mu\left(1 \right)) \nonumber
	\end{align} 	
	as $R\to\infty$.  The constant $c(\mathcal{S},\mu)$ will be positive if	there is a function $F_0$ with a regular (i.e., the gradient doesn't vanish) connected component in $\mathcal{S}$ contained in $B(r)$ for some $r>0$
	and $F_0\in \overline{\mathcal F\, L^2_{\tt H}(\mu)}^{C^s}$, similarly for $c(T,\mu)$.
\end{thm}
The last condition means that $F_0$ can be approximated in $C^s(K)$, for $K$ any compact set, by functions in $\mathcal{F} L^2_H(\mu)$.
\begin{proof}
	Let $r>0$,	we define:
	$$
	\Phi_r^\textnormal{a}(G) \coloneqq \frac{\mathcal{N}(G,\cdot,r)}{\vol B(r)}, \qquad
	\Psi_r^\textnormal{a}(G) \coloneqq \frac{\mathcal{NI}(G,\cdot,R)}{\vol B(r)}\,,
	$$	
	where $\mathcal{NI}(G,\cdot,R)$ denotes de number of nodal domains intersecting the boundary of $S(R)$.
	Since $F_{\mu}$ is translation invariant by Bochner's Theorem,  Wiener's Ergodic Theorem  \cite[Section 6]{NS} implies\footnote{We can apply the Ergodic Theorem, despite our field might not be ergodic (by Fomin-Grenander-Maruyama Theorem, see, e.g., \cite{NS}, as $\mu$ might have atoms) because we only need the translational invariance.} that
	\begin{align}
		&\frac1{\vol B(R)}\int_{B(R)} \Phi_r^\textnormal{a}(\tau_v \cdot)\, dv\to \bar{\Phi}^\textnormal{a}_r  & R\rightarrow \infty \label{2.2.1}
	\end{align}
	a.s. and in $L^1$,  where $\tau_v G\coloneqq G(\cdot+v)$ for $v\in \R^m$. Moreover, $\bar{\Phi}^\textnormal{a}_r$ is invariant under $\tau_v$, $\bE[\bar{\Phi}^\textnormal{a}_r]=\bE[{\Phi}^\textnormal{a}_r]$ and similarly for $\Psi_r^\textnormal{a}$.
	
	Thanks to the integral-geometric sandwich (\cite[Lemma 1]{NS}), and following the proof of \cite[Theorem 1]{NS}, see also the proof of Proposition \ref{semi-locality m>2}, we have that \eqref{2.2.1} implies that the limit
	\[
	c(G,\cdot,\mu) \coloneqq \lim_{R\to\infty} \frac{\mathcal{N}(G,\cdot,R)}{\vol B(R)}
	\]
	exists a.s. and in $L^1$. Note that it is not a constant but a random variable, thus letting
	$c(\cdot,\mu)\coloneqq \bE\left(c(F_\mu,\cdot,\mu)\right),$ the first statement of the theorem follows from the $L^1$ convergence. 
	Let us now consider the positivity of the constants. From \eqref{2.2.1} and the integral-geometric sandwich, we have 
	\begin{equation}\nonumber
		\bE\left(\Phi_r^\textnormal{a}\right) \le c(\cdot,\mu) \le 	\bE\left(\Phi_r^\textnormal{a}\right) + 	\bE\left(\Psi_r^\textnormal{a}\right).
	\end{equation}	
	Thus, in order to prove that $c(\cdot,\mu)>0$, thanks to Chebyshev's inequality, it is enough to show that 
	\begin{align}\label{claim positivity}
			\bP\big(\big\{ F_\mu\in C^s(\R^m): \mathcal{N}(F_\mu,\mathcal{S},r)\ge 1\big\}\big)>0. 
	\end{align}
 Let $F_0$ be as in the statement of the theorem, by \cite[Appendix A.7, A.12]{NS} for $s=0$ and \cite[Proposition 3.8]{EPRBeltrami} for general $s$, $F_0$ is in the support of the measure on the space of $C^s$ functions of our random field $F_\mu$, that is, for any compact set $K\subset\R^m$ and
	 each~$\varepsilon>0$,
	 \begin{equation}\label{Support}
	 	\bP\big(\big\{ F_\mu\in C^s(\R^3): \|F_\mu-F_0\|_{C^s(K)}<\varepsilon\big\}\big)>0\,.
	 \end{equation} 
	Now, as the connected component of $F_0$ in $\mathcal{S}$ is regular by hypothesis, we can apply Thom's Isotopy, Theorem \ref{ThThom} below, to conclude that if 
	\begin{align}
	\|F-F_0\|_{C^s(K)}<\delta, \label{3.1}
	\end{align}
	where the connected component of $F_0$ is in the interior of $K$, then $F$ also has a connected component diffeomorphic to $\mathcal{S}$. Finally \eqref{claim positivity} follows from \eqref{Support}, taking $\varepsilon=\delta$ in \eqref{3.1}. 	We can proceed similarly for nesting trees and conclude the proof.  	
\end{proof}

\begin{exmp}\label{RemPosCte}
	If $\mu=\sigma_{m-1}$, the Lebesgue measure on the sphere, then it is enough to show $F_0$ is a solution to the Helmholtz equation as this set equals $\overline{\mathcal F\, L^2_{\tt H}(\sigma)}^{C^s}$ \cite[Proposition 6]{CS19}.  However, in this case, the construction of the particular functions for topological classes gives $F_0$ as a (finite) sum of the form \cite{EPS13,CS19}
	\begin{equation*}
		F_0(x)=(2\pi)^{\frac n2} \sum_{l=0}^L\sum_{m=1}^{d_l} a_{lm} \,
		Y_{lm}\bigg(\frac x{|x|}\bigg) \,\frac{J_{l+\frac n2-1}(|x|)}{|x|^{\frac n2-1}}\,,
	\end{equation*}
	so by \cite[Proposition 2.1]{EPR19} or by
	Herglotz Theorem~\cite[Theorem 7.1.28]{Hor15} and the rapid decay of Bessel functions, $F_0\in{\mathcal F\, L^2_{\tt H}(\sigma)}$. For instance, the example mentioned above could be the spherical Bessel functions
	$$
	C_m\frac{J_{\frac n2-1}(|x|)}{|x|^{\frac n2-1}}=\int_{\S^{m-1}}e^{i\braket{x}{\omega}}d\sigma_{m-1}(\omega)
	$$
	where $C_m$ is as in \eqref{eq:kernel intro}. They are radial solutions to the Helmholtz equation, so the nodal sets are spheres with the radii the zeros of $J_{\frac n2-1}(|x|)$. See Figure \ref{FigNod1} for the case of $n=2$. This proves $c_{NS}(\sigma_{m-1})\coloneqq c(H(m-1),\sigma_{m-1})>0$ as $c(\sigma_{m-1},\{[\S^{m-1}]\})>0$. See also \cite[Condition ($\rho4$), Appendix C]{NS} for sufficient conditions to ensure $c_{NS}(\mu)>0$ and \cite{RI18} for an explicit lower bound together with some numerical estimates
\end{exmp}
The stability property of the nodal set used above is given by the following theorem.
\begin{thm}[Compact Thom's Isotopy Theorem]\label{ThThom}
	Let $V$ be an domain in $\R^m$ and let $h:V\to\R$ be a $C^\infty$ map. Consider a (compact) connected component $L\subset\subset V$ (i.e., which is compactly embedded in $V$) of the zero set $h^{-1}(0)$ and suppose that:
	$$
	|\nabla h\vert_{L}|>0.
	$$
	Then, given any $\varepsilon>0$ and $p\ge 1$, there exists some $U\subset\subset V$ neighbourhood of $L$ and $\delta>0$ such that for any smooth function $g:U\to\R^m$ with
	\begin{equation} \nonumber
		\| h- g\|_{C^p(U)}<\delta
	\end{equation}
	one can transform $L$ by a diffeomorphism $\Phi$ of $\R^m$ so that $\Phi(L)$ is the intersection of the zero set $g^{-1}(0)$ with $U$. The diffeomorphism $\Phi$ only differs from the identity in a proper subset of $U$ (i.e., a subset $\subsetneq U$) and satisfies $\|\Phi-\textnormal{id}\|_{C^p(\R^m)}<\varepsilon$.
\end{thm}
The proof follows from \cite[Theorem 3.1]{EPS13}, we reproduce some details for completeness. 
\begin{proof} We have to construct a domain $U$ and find some $\eta>0$ such that the component of $h^{-1}(B(0,\eta))$ connected with $L$ is contained in $U$ and $\inf_U\norm{\nabla h}>0$. For this purpose, let us define the following vector field:
	$$X(x):=\frac{\nabla h(x)}{\norm{\nabla h(x)}^2} $$
	which is well defined if the gradient does not vanish. Denote by $\varphi^t$, the associated flow, that is, the solution to $\partial_t\varphi^t(x)=X(\varphi^t(x))$.	Considering the derivative with respect to time, if $h(x)=0$ then $h(\varphi^t(x))=t,$ and 
	\begin{equation}\label{DerS}
	\norm{\partial_t\varphi^t(x)}	=\norm{X(\varphi^t(x))}=\frac{1}{\norm{\nabla h(\varphi^t(x))}}.
	\end{equation}
	By compactness and regularity of the connected component, $\norm{\nabla h\vert_{L}}\in [c, C]$, with $c>0$. Since $\varphi(t,x)\coloneqq \varphi^t(x) $ is a smooth map, if we define $H:\R\times L$ as $H(t,x)\coloneqq \norm{\nabla h(\varphi(t,x))}$, then $H^{-1}(c-\delta,\infty)$ is an open set of
	$\R\times L$, for any $\delta>0$,  and it includes $\{0\}\times L$. By compactness and the product topology, there exists a finite number of $t_i>0$, $U_i^L$ open sets of $L$ (induced topology) such that
	$$
	\norm{\nabla h(\varphi(t,x))}>c_1
	$$
	for $t\in(-t_i,t_i)$, $x\in U_i^L$ and $c_1\coloneqq c-\delta.$ If we define $\eta\coloneqq \frac{1}{2}\min\{t_i,c_1 d\},$ where $d\coloneqq \text{dist}(L,\partial V)$, then we claim that $U\coloneqq L_{+\eta/c_1}$ is the desired neighbourhood. Indeed, if $y$ is the component of $h^{-1}(B(0,\eta))$ connected with $L$, then $y=\varphi^{\eta'}(x)$ with $\eta'<\eta$, $x\in L$ so by \eqref{DerS} and Lagrange Theorem
	$$
	\norm{y-x}=\|\varphi^{\eta'}(x)-x\|\le \frac{\eta}{c_1},
	$$	
	hence, $y\in B(x,\eta/c_1)\subset U$. Furthermore, if $y\in U$, then $\forall~v\in\partial V$
	$$
	\norm{y-v}\ge\norm{x-v}-\norm{y-x}\ge d-d/2>0,
	$$
	where $y\in B(x,\eta/c_1)$ and $\eta/c_1<d/2$ by definition.
\end{proof}

\section{Bourgain's de-randomisation, proof of Theorem \ref{thm 1}.}
\label{proof thm 1}
The content of this section follows closely the proofs in \cite{BU,BW} to extend the ideas from $\mathbb{T}^2$ to $\R^m$. 
\subsection{The function $\phi_x$}
\label{notation}
Let $m\geq 2$ be fixed, $R\gg W>1$ be in section \ref{stament main reuslts}. Using hyperspherical coordinates, that is, writing $x\in \S^{m-1}$ as $x=G(\theta)$ where
$$
G(\theta)\coloneqq(\cos\pi\theta_1, \sin \pi\theta_2\cos \pi\theta_2,...,\\ \sin\pi\theta_1\cdots\sin\pi\theta_{m-2}\sin2\pi\theta_{m-1})
$$
such that $G\lvert_{(0,1)^{m-1}}$ is a diffeomorphism onto $\S^{m-1}\backslash S'$, where $S'$ is a set of measure zero, we identify $\S^{m-1}$ with $[0,1]^{m-1}$.  Now, let $K>1$ be a (large) parameter and divide $[0,1]^{m-1}$ into $K^{m-1}$ cubes and use hyper-spherical coordinates to divide the sphere into $K^{m-1}$ regions which we call $I_k$.   Let $\{\zeta^k\}\subset \S^{m-1}$ be the \textquotedblleft centres \textquotedblright of such regions (centre is defined again picking the centre in $[0,1]^{m-1}$ and projecting onto the sphere using hyper-spherical coordinates).  Finally, pick another parameter $\delta>0$ and let  $\mathcal{K}$ to be the set of $k$'s such that $k\in \cK$ if and only if
\begin{align}\label{MassDelta}
\mu_r( I_k)>\delta.
\end{align}
We will need the following two simple properties of this partition: 
\begin{claim} 
	\label{useful claim}We  have the following: 
	\begin{enumerate}
		\item[\textnormal{i)}] $\sum_{k\in \mathcal{K}}\mu_r(I_k)= 1- \sum_{k\not\in \mathcal{K}}\mu_r(I_k)= 1 + O(\delta K^{m-1})$
		\item[\textnormal{ii)}] If $r_n\in I_k$, then $\norm{r_n-\zeta^k}=O(K^{-1})$. 
	\end{enumerate}
\end{claim}
\begin{proof}
	i) follows from the fact that there are at most $K^{m-1}$ elements in the complement of $\cK$. ii) follows from the fact that $G\vert_{[0,1]^{m-1}}$ is a smooth function so it is Lipschitz and, writing $G(\theta_k)=\zeta^k$, we have 
	$$
	\norm{G(\theta)-G(\theta_k)}\le C_G\norm{\theta-\theta_k}.
	$$
\end{proof}
 As $r_n=-r_{-n}$, in order to count only one these points, we define $\cK^+$ as the set of $k\in\cK$ such that $(\zeta^k)_j>0$ with $j\coloneqq\max_{1\le i\le m}\{(\zeta^k)_i\neq 0\}$, where $(\zeta^k)_i$ denotes the $i$-th component of $\zeta^k$. Note that, by definition, 
\begin{align}\label{MassInfty}
 2N\delta\le 2N\mu_r(I_k)=\#\{|n|\le N\hspace{2mm}/\hspace{1mm}r_n\in I_k\}\rightarrow \infty \quad\text{ as } N\rightarrow \infty. 
\end{align}
 Finally, we define the auxiliary function, as in \eqref{eq:def phi}
\begin{align}\label{phi}
	\phi_x(y)\coloneqq \sum_{k\in\cK} \left[\frac{1}{\left(2N\mu_r(I_k)\right)^{1/2}}\sum_{r_n\in I_k} a_n e(r_n\cdot x)\right]{\mu_r(I_k)^{1/2}}e\left( \zeta^k\cdot y\right). 
\end{align}
The next lemma shows that $\phi_x$ is, on average, a good approximation of $F_x$.                    
\begin{lem}
	\label{first approx}
Let $F_x$ and $\phi_x$ be as in \eqref{Fx} and \eqref{phi} respectively, $R\gg W>1$ and $K,\delta>0$ be as in Section \textnormal{\ref{notation}} with $\delta<K^{-m+1}$, $s\geq 0$ be some integer and $l=\lfloor\frac{m}{2}+1\rfloor$. Then, we have  
	\begin{align}                                                   \Bint_{B(R)}\norm{F_x- \phi_x}^2_{C^s(B(W))} \lesssim W^{2(s+l)+m}\left(\delta K^{m-1} + W^2K^{-2}\right)\left(1+O_N(R^{-\Lambda-3/2}) \right) \nonumber.
	\end{align}	

\end{lem}
   
\begin{proof}
	Using Sobolev's Embedding Theorem, we bound the ${C}^s(B(W))$ norm by the $H^{s+l}(B(W))$ norm, and rescaling to a ball or radius one, we obtain  
	\begin{align}
	\Bint_{B(1)}	\norm{F_{Rx}-\phi_{Rx}}^2_{C^s(B(W))}dx \lesssim \sum_{|\alpha| \leq s+l}\Bint_{B(1)} \norm{D^{\alpha}(F_{Rx}-\phi_{Rx})}_{L^2(B(W))}^2dx  \nonumber
	\end{align}
	where $D^{\alpha}$ is the multi-variable derivative. If $\alpha=0$,  denoting $\S^{m-1}_*\coloneqq\mathbb{S}^{m-1}\backslash \bigcup_{k\in\cK} I_k$ and rescaling the ball of radius $W$ to a ball of radius $1$, we have
	\begin{align}\label{4}
	&\Bint_{B(1)} \norm{(F_{Rx}-\phi_{Rx})}^2_{L^2(B(W))}dx \lesssim W^m\Bint_{B(1)} \norml{\frac{1}{2N} \sum_{r_n\in \S^{m-1}_*} a_n e(\braket{r_n}{R x})e\left(\langle r_n, Wy\rangle\right)}^2dx  \nonumber \\
	&+ W^m\Bint_{B(1)} \norml{ \sum_{k\in \mathcal{K}} \left(2N\right)^{-1/2}\sum_{r_n\in I_k}a_n e(\braket{r_n}{R x})\left(e\left(\langle r_n, Wy\rangle\right)- e\left(\langle \zeta ^k, Wy\rangle\right)\right)}^2 dx.
	\end{align}
 To evaluate the integrals in \eqref{4}, we will need the following claim: 
		\begin{align}\label{5}
\Bint_{\!B(1)} e(\langle r_n-r_{n'},Rx\rangle)dx= \begin{cases}
1 & n=n' \\
C_m\dfrac{J_{\Lambda+1}(2\pi R\norm{r_n-r_{n'}})}{(R\norm{r_n-r_{n'}})^{\Lambda+1}} & n \neq n' 
\end{cases} = \delta_{n,n'} +O_N(R^{-\Lambda-3/2}) 
\end{align}
Indeed, by the Fourier Transform of spherical harmonics \cite[Proposition 2.1]{EPR19}:  
	\begin{equation}\label{FouSH}
	\int_{\S^{m-1}}Y_{\ell
			}e^{i \braket{x}{\cdot}}d\sigma=(2\pi)^{\frac{m}{2}}\,(-i)^\ell\, Y_{\ell} \left(\frac x{|x|}\right) \frac{J_{\ell + \Lambda}(|x|)}{|x|^{\Lambda}}\quad \Lambda\coloneqq\frac{m-2}{2},
	\end{equation}
	where $\ell$ is the index associated with the eigenvalue and $J_\alpha$ represents the Bessel function of first order and index $\alpha$. Setting $\ell=0$ in \eqref{FouSH} and using polar coordinates:
	\begin{equation}
	\label{3.1.1}
	\int_{B(1)} e^{i2\pi \braket{x}{y}}dy=(2\pi)\int_0^1 r^{m/2} \frac{J_{\Lambda}(r 2\pi|x|)}{|x|^\Lambda} \, d r=\frac{J_{\Lambda+1}(2 \pi  |x|)}{|x|^{\Lambda+1}}.
	\end{equation}
	Moreover, by the standard asymptotic expansion of Bessel functions \cite[Chapter 7]{Wbook}:
	\begin{equation}\label{jal}
	J_\alpha(z) = \sqrt{\frac{2}{\pi z}} \, \cos\bigg(z-\frac{\pi \alpha}{2}-\frac{\pi}{4}\bigg) +  O_\alpha(z^{-3/2})
	\end{equation}
	Thus, for $n\neq n'$, using  \eqref{3.1.1} and \eqref{jal}, and bearing in mind \eqref{1},  \eqref{5} follows upon noticing that 
	$$
	J_{\Lambda+1}(2\pi R\norm{r_n-r_{n'}})=O\left((R\norm{r_n-r_{n'}})^{-1/2}\right)=O_N(R^{-1/2}).
	$$

	In order to bound the first term of the RHS of \eqref{4}, we expand the square, use Fubini and \eqref{5} to obtain 
\begin{align}\label{5.1}
	&\frac{W^m}{2N}	\int_{B(1)}dy\Bint_{\!B(1)}dx\sum_{r_n\in \S^{m-1}_*}\sum_{r_{n'}\in \S^{m-1}_*} a_n \overline{a}_{n'} e(\braket{r_n-r_{n'}}{R x})e\left(\langle r_n-r_{n'}, Wy\rangle\right)=\nonumber\\
	&=\frac{W^m}{2N}\int_{B(1)}dy	\sum_{r_n\in \S^{m-1}_*} |a_n|^2+ O_N\left( R^{-\Lambda-3/2}\right) \sum_{r_n\in \S^{m-1}_*}\sum_{r_{n'}\in \S^{m-1}_*}a_n \overline{a}_{n'} e\left(\langle r_n-r_{n'}, Wy\rangle\right),
\end{align}
with $r_{n'}\neq r_n$ in the second summand. Since $|a_n|=1$, bearing in mind \eqref{MassDelta} and using Claim \ref{useful claim}, we can bound \eqref{5.1} by
\begin{align}
\text{RHS}\eqref{5.1}\lesssim W^m \delta K^{m-1}\left(1+O_N(R^{-
	\Lambda-3/2})\right). \label{5.3}
\end{align}	
For the second summand of the RHS of \eqref{4} we proceed similarly, taking into account Claim \ref{useful claim}, we have
	$$
	|e\left(\langle r_n, Wy\rangle\right)- e\left(\langle \zeta ^k, Wy\rangle|\right)\le \norm{r_n-\zeta ^k}\norm{W y}\lesssim \frac{W}{K}.
	$$
	Thus, expanding the square and using  Fubini and \eqref{5}, we can bound the second term on the right hand side of \eqref{4} as 
	\begin{align}
		&\Bint_{B(1)} \norml{ \sum_k \left(2N\right)^{-1/2}\sum_{r_n\in I_k}a_n e(\braket{r_n}{R x})\left(e\left(\langle r_n, Wy\rangle\right)- e\left(\langle \zeta ^k, Wy\rangle\right)\right)}^2 dx\lesssim \nonumber\\
		&\lesssim \sum_k \left(2N\right)^{-1}\sum_{r_n\in I_k} |1-e\left(\langle r_n-\zeta^k, Wy\rangle\right)|^2+|\sum_{r_n\neq r_{n'}}	C_m\dfrac{J_{\Lambda+1}(2\pi R\norm{r_n-r_{n'}})}{(R\norm{r_n-r_{n'}})^{\Lambda+1}}\times\nonumber\\
		&\left.\times (2N)^{-1} a_n\overline{a}_m \left(e\left(\langle r_n, Wy\rangle\right)- e\left(\langle \zeta ^k, Wy\rangle\right)\right) \left(e\left(\langle r_{n'}, Wy\rangle\right)- e\left(\langle \zeta ^{k'}, Wy\rangle\right)\right)\right|\lesssim \nonumber\\
		&\lesssim\hspace{1mm} W^2K^{-2}(1+ O_N(R^{-\Lambda-3/2})). \label{5.4}
	\end{align}
	All in all, using \eqref{5.3} and \eqref{5.4},  we obtain 
	\begin{align}
	\Bint_{B(1)} \norm{(F_{Rx}-\phi_{Rx})}_{L^2(B(1))}dx \lesssim W^m\left(\delta K^{m-1} + W^2K^{-2}\right)\left(1+ O_N(R^{-\Lambda-3/2})\right). \nonumber
	\end{align}
	For $\alpha \neq 0$, observe that if we differentiate with respect to $x$,
	$$
	D^{\alpha}(e(y,x))=(2\pi i)^{|\alpha|}\left(\prod_{i=1}^m y_i^{\alpha_i}\right)e(y,x).
	$$
	Thus,
	$$
	D^{\alpha}(e\left(\langle r_n, Wy\rangle\right))=(2\pi W)^{|\alpha|} e\left(\langle r_n, Wy\rangle\right)\prod_{i=1}^m r_{n,i}^{\alpha_i}.
	$$
	Also,	
	\begin{equation*}
	|D_\alpha\left(e\left(\langle r_n, Wy\rangle\right)- e\left(\langle \zeta ^k, Wy\rangle\right)\right)|=
	(2\pi W)^{|\alpha|} \abs{ \prod_{i=1}^m r_{n,i}^{\alpha_i}-\prod_{i=1}^m (\zeta_{i}^k)^{\alpha_i}e\left(\langle \zeta ^k-r_n, Wy\rangle\right)}.
	\end{equation*}
	Now, adding and subtracting $\prod_{i=1}^m (\zeta_{i}^k)^{\alpha_i}$ and using the triangle inequality, gives:
	\begin{align*}
	&|D_\alpha\left(e\left(\langle r_n, Wy\rangle\right)- e\left(\langle \zeta ^k, Wy\rangle\right)\right)|\le (2\pi W)^{|\alpha|} \Bigg( \abs{ \prod_{i=1}^m r_{n,i}^{\alpha_i}-\prod_{i=1}^m (\zeta_{i}^k)^{\alpha_i}}+\\
	&\quad+\abs{\prod_{i=1}^m (\zeta_{i}^k)^{\alpha_i}}\abs{ 1- e\left(\langle \zeta ^k-r_n, Wy\rangle\right) } \Bigg)
		\end{align*}
	Since $|e^{ix}-e^{iy}|\le|x-y|$, we bound the last expression by:
	\begin{equation}\label{5.2}
	W^{|\alpha|}\left(\abs{\zeta^{k}-r_n}+W\abs{\zeta^{k}-r_n}\right) \lesssim\frac{W^{|\alpha|+1}}{K}.
	\end{equation}
	Hence, following a similar argument as in the case $\alpha=0$, combined with \eqref{5.2}, we conclude that
	{\small \begin{align}
	\int_{B(1)} \norm{D^{\alpha}(F_{Rx}-\phi_{Rx})}^2_{L^2(B(1))}dx \lesssim W^{2|\alpha|+m}\left(\delta K^{m-1} + W^2K^{-2}\right)\left(1+O_N(R^{-\Lambda-3/2}) \right) \nonumber
 	\end{align}}
 finishing the proof.
\end{proof}

\subsection{Gaussian moments}
\label{Gaussian moments}
Let us define:
$$b_k(x)\coloneqq\frac{1}{(2N\mu_r(I_k))^{1/2}}\sum_{r_n\in I_k} a_n e(r_n\cdot R x).$$ 
 We are going to show that the pseudo-random vector $(b_k)_{k\in \mathcal{K}}$  approximates a Gaussian vector $(c_k)_{k\in \mathcal{K}}$, where $c_k$ are i.i.d. complex standard Gaussian random variables subject to $\overline{c}_k=c_{-k}$. More specifically, we prove the following quantitative lemma:
\begin{lem}
	\label{moments}
Let  $N\geq 1$, $R>0$, $b_k$ be as above and $\cK^+,K,\delta$ be as in Section \textnormal{\ref{notation}}. Moreover, let $D>1$ be some large parameter and fix two sets of positive  integers $\{s_k\}_{k\in \cK^+}$ and $\{t_k\}_{k\in \cK^+}$  such that $ \sum_k s_k+t_k \leq D $, then we have 
\begin{align}
\left|\Bint_{B(R)}\left[\prod_{k\in \cK^+}b_k^{s_k}\overline{b}_k^{t_k}\right]- \mathbb{E}\left[\prod_{k\in \cK^+}c_k^{s_k}\overline{c}_k^{t_k}\right]\right|= O_D(\delta^{-1}N^{-1}) + O_{N,D}(R^{-\Lambda -3/2}). \nonumber
\end{align}

\end{lem}    
\begin{proof}
	For $c_k$, first note that the independence properties of the Gaussian variables~$c_k$ (which have zero mean)
	imply that $\bE
	(c_k\overline{c_{k'}})=0$ if $k'\not\in \{k,-k\}$. When
	$k'=k$ one has
	\[
	\bE [|c_k|^2]= \bE[ (\textnormal{Re}(c_k))^2]+ \bE[ (\textnormal{Im} (c_k))^2]=1\,,
	\]
	and when $k'=-k$,
	\[
	\bE [(c_k)^2]= \bE[ (\textnormal{Re} (c_k))^2]- \bE[ (\textnormal{Im}
	(c_k))^2]+2i \,\bE[ (\textnormal{Re} (c_k))(\textnormal{Im} (c_k))]=0\,.
	\]
	Therefore, 
	\begin{equation}\label{UncorrelationCoeff}
	\bE	(c_k\overline{c_{k'}})=\delta_{k k'}.
	\end{equation}	
	Thus by independence and  a similar calculation for $\bE\left[ c_k^{t_k}\bar{c}_k^{s_k}\right]$,
	\begin{equation}\label{NormalMom}
	\bE\left[\prod_{k\in\cK^+} c_k^{t_k}\bar{c}_k^{s_k}\right]=\prod_{k\in\cK^+} \bE\left[ c_k^{t_k}\bar{c}_k^{s_k}\right]=\prod_{k\in\cK^+} \delta_{t_k,s_k}s_k!.     
	\end{equation}
	For the moments of $b_k$, we can prove the following:
	\begin{claim}\label{claim:moment bk} For $k\in\cK^+$ we have
		$$
		\Bint_{B(R)}\left[b_k^{s_k}\overline{b}_k^{t_k}\right]=	\delta_{s_kt_k}\left(s_k!+O_{s_k}\left((\delta N)^{-1}\right)\right)+O_{D,N}\left(R^{-\Lambda-3/2}\right).
		$$  
	\end{claim}
	\begin{proof}[Proof of Claim \textnormal{\ref{claim:moment bk}}]
		We have that:
		$$
		b_k^{s_k}\overline{b}_k^{t_k}=|I_k|^{-(s_k+t_k)/2}\sum_{\mathcal{C}_k}\prod_{i=1}^{s_k}\prod_{j=1}^{t_k} a_{i,k}\overline{a}_{j,k} e(\braket{r_{i,k}-r_{j,k}}{Rx})
		$$
		where $r_{i,k}\coloneqq r_{n_{i,k}}\in I_k$, $a_{i,k}\coloneqq a_{n_{i,k}}$ and $\mathcal{C}_k$ represents the set of all possible choices of $n_{i,k}$ and $n_{j,k}$ with $i\in\{1,...,s_k\}$ and $j\in\{1,...,t_k\}$. Then, rescaling to a ball of radius $1$, we have 
		\begin{align}\label{MomentB}                                                             
			&\Bint_{B(R)}\left[\prod_{k\in \cK^+}b_k^{s_k}\overline{b}_k^{t_k}\right]= \left(\prod_{k\in\cK^+}|I_k|^{-(s_k+t_k)/2}\right)\sum_{\mathcal{C}}\prod_{k\in\cK^+} \prod_{i=1}^{s_k}\prod_{i=1}^{t_k} a_{i,k}\overline{a}_{j,k} \times\nonumber \\ 
			&\times \Bint_{B(1)} e\left(\sum_{k\in\cK^+}\left(\sum_{i=1}^{s_k} r_{i,k}-\sum_{j=1}^{t_k} r_{j,k}\right)\cdot Rx\right) dx
		\end{align}
		where $\mathcal{C}$ represents the set of all possible choices of $n_{i,k}$ and $n_{j,k}$ with $i\in\{1,...,s_k\}$, $j\in\{1,...,t_k\}$ and $k\in\cK^+$. To estimate this, let us begin by fixing $k$:
		\begin{align*}
			&\Bint_{B(R)}\left[b_k^{s_k}\overline{b}_k^{t_k}\right]= |I_k|^{-(s_k+t_k)/2}\sum_{\mathcal{C}_k} \prod_{i=1}^{s_k}\prod_{i=1}^{t_k} a_{i,k}\overline{a}_{j,k} \Bint_{B(1)}e\left(\left(\sum_{i=1}^{s_k} r_{i,k}-\sum_{j=1}^{t_k} r_{j,n}\right)\cdot Rx\right) dx.
		\end{align*}
		In the inner sum, 
		$$\sum_{i=1}^{s_k} r_{i,k}-\sum_{j=1}^{t_k} r_{j,k}=\sum_{r_n\in I_k}\alpha_nr_n-\sum_{r_n\in I_k}\beta_nr_n=\sum_{r_n\in I_k}\gamma_nr_n,$$
		with $\alpha_n,\beta_n,\gamma_n$ integers. By rational independence \eqref{1}, the sum vanishes if and only if  $\gamma_n=0$ for every $n$. So $\mathcal{C}_k$ can be divided into the combinations where $\gamma_n=0$ for every $n$, $\mathcal{C}_{k,1}$, and the remaining terms, $\mathcal{C}_{k,2}$:
		\begin{align}
			&\Bint_{B(R)}\left[b_k^{s_k}\overline{b}_k^{t_k}\right]=|I_k|^{-(s_k+t_k)/2}\Bigg(\sum_{\mathcal{C}_{k,1}} 1+ \nonumber\\
			&+\sum_{\mathcal{C}_{k,2}} \prod_{i=1}^{s_k}\prod_{i=1}^{t_k} a_{i,k}\overline{a}_{j,k} \Bint_{B(1)}e\left(\left(\sum_{i=1}^{s_k} r_{i,k}-\sum_{j=1}^{t_k} r_{j,n}\right)\cdot Rx\right)dx\Bigg). \label{5.5}
		\end{align}
		For the first term, note that if $\gamma_n=0$, then $\alpha_n=\beta_n$, thus $s_k=\sum_{r_n\in I_k}\alpha_n=\sum_{r_n\in I_k}\beta_n=t_k.$
		Then, we denote by $d_k\coloneqq\#\{n~|~\alpha_n\neq 0\}$ and write $
		\mathcal{C}_{k,1}=\bigsqcup_{d_k=1}^{s_k}\mathcal{C}_{k,1,d_k}
		$
		where $\mathcal{C}_{k,1,d_k}$ the set of indexes of $\mathcal{C}_{k,1}$ where the number of $\alpha_n\neq 0$ equals $d_k$. Thus, the first term on the right hand side of \eqref{5.5} is
		\begin{align}
			|I_k|^{-(s_k+t_k)/2}\Bigg(\sum_{\mathcal{C}_{k,1}} 1\Bigg)=\delta_{t_k,s_k}\frac{\sum_{d_k=1}^{s_k}\#\mathcal{C}_{k,1,d_k}}{|I_k|^{s_k}} \label{5.6.1}
		\end{align}
		Assume now that we fix the set of $\{\alpha_n\}$ and $d_k$, let us calculate the number of possible indexes in $\mathcal{C}_{k,1,d_k}$  the number of $\alpha_n\neq 0$ equals $d_k$. If we assume that $N$ is large enough such that $|I_k|\ge \max s_l$, which is possible by \eqref{MassInfty}, we may write
		$\mathcal{C}_{k,1,d_k}=\frac{s_k!}{\prod_{n\in\mathcal{J}_k}\alpha_n!},$
		where $\mathcal{J}_k$ is an index set associated with the $r_n\in I_k$.	Thus, 
		\begin{equation}\label{CardIdk}
			\#\mathcal{C}_{k,1,d_k}=\sum \left(\frac{s_k!}{\prod_{n\in\mathcal{J}_k}\alpha_n!}\right)^2
		\end{equation}
		where the sum runs over the possible combination of $\alpha_n$ such that $d_k=\#\{n~|~\alpha_n\neq 0\}$ and $\sum_{n\in\mathcal{J}_k}\alpha_n=s_k.$ If $d_k=s_k$, then $\alpha_k\le 1$ and
		$$
		\frac{	\#\mathcal{C}_{k,1,d_k}}{|I_k|^{s_k}}=\frac{s_k!^2}{\left(2N\mu_{r}(I_k)\right)^{s_k}}\frac{(2N\mu_{r}(I_k))!}{\left(2N\mu_{r}(I_k)-s_k\right)!s_k!}=s_k! (1-\varepsilon_{s_k,k}),
		$$
		as there are $\displaystyle\binom{|I_k|}{s_k}$ ways of choosing the elements with $2N\mu_{r}(I_k)=|I_k|$. We have that $\varepsilon_{s_k,k}>0$  will be:
		\begin{align*}
			1-\prod_{i=0}^{s_k-1}\left(1-\frac{i}{2N\mu_{r}(I_k)}\right)\le 1- &\left(1-\frac{s_k}{2N\mu_{r}(I_k)}\right)^{s_k}\\
			&=\sum_{i=1}^{s_k}\binom{s_k}{i}\left(\frac{s_k}{2N\mu_{r}(I_k)}\right)^i\lesssim_{s_k}(2N\delta)^{-1}.
		\end{align*}

		Now, consider the case when $1\le d_k<s_k$, the unlabelled sum in \eqref{CardIdk} will have
		$\displaystyle\binom{2N\mu_{r}(I_k)}{d_k}$ elements, thus it can be bounded by
		\begin{align*}
			|\#\mathcal{C}_{k,1,d_k}|< \frac{s_k!^2}{\left(2N\mu_{r}(I_k)\right)^{s_k}}&\frac{2N\mu_{r}(I_k)}{\left(2N\mu_{r}(I_k)-d_k\right)!d_k!}=\\
			&=\frac{s_k!^2}{\left(2N\mu_{r}(I_k)\right)^{s_k-d_k}d_k!}\frac{2N\mu_{r}(I_k)!}{\left(2N\mu_{r}(I_k)-d_k\right)!\left(2N\mu_{r}(I_k)\right)^{d_k}},
		\end{align*}
		with
		$$                                                                                        
		\frac{\left(2N\mu_{r}(I_k)\right)!}{\left(\left(2N\mu_{r}(I_k)\right)-d_k\right)!\left(2N\mu_{r}(I_k)\right)^{d_k}}=(1-\varepsilon_{d_k,k})
		$$
		and $\varepsilon_{d_k,k}= O_{d_k}(N\delta)^{-1}$. Therefore, the first term on the right hand side of \eqref{5.5} via \eqref{5.6.1} is 
		\begin{equation}\label{MomentFactorial}
			|I_k|^{(-s_k+t_k)/2}\sum_{\mathcal{C}_{k,1}}1 = 	\delta_{s_kt_k}\left(s_k!+O_{s_k}\left((\delta N)^{-1}\right)\right).
		\end{equation}	
		For the second term, by construction, the inner sum does not vanish, so:
		\begin{align*}
			\sum_{\mathcal{C}_{k,2}} \prod_{i=1}^{s_k}\prod_{i=1}^{t_k} a_{i,k}\overline{a}_{j,k} \Bint_{B(1)}&e\left(\left(\sum_{i=1}^{s_k} r_{i,k}-\sum_{j=1}^{t_k} r_{j,n}\right)\cdot Rx\right)dx=\\
			&=\sum_{\mathcal{C}_{k,2}} \prod_{i=1}^{s_k}\prod_{i=1}^{t_k} a_{i,k}\overline{a}_{j,k}C_m\dfrac{J_{\Lambda+1}\left(2\pi R\norm{\sum_{i=1}^{s_k} r_{i,k}-\sum_{j=1}^{t_k} r_{j,n}}\right)}{\left(R\norm{\sum_{i=1}^{s_k} r_{i,k}-\sum_{j=1}^{t_k} r_{j,n}}\right)^{\Lambda+1}}, 
		\end{align*}
		by \eqref{5}. Using \eqref{5} again and \eqref{1}, the second term is $O_{D,N}\left(R^{-\Lambda-3/2}\right)$. Thus, via  \eqref{5.5} and \eqref{MomentFactorial}, we finally obtain:
		$$
		\bE\left[b_k^{s_k}\overline{b}_k^{t_k}\right]=	\delta_{s_kt_k}\left(s_k!+O_{s_k}\left((\delta N)^{-1}\right)\right)+O_{D,N}\left(R^{-\Lambda-3/2}\right).
		$$  
	\end{proof} 
	
	Similarly, we claim that, if $k\neq k'$, then
	\begin{equation}\label{MomentsIndep}
	\Bint_{B(R)}\left[b_k^{s_k}\overline{b}_{k'}^{t_{k'}}\right]=O_{D,N}\left(R^{-\Lambda-3/2}\right).
	\end{equation}
	Indeed, as $k\neq k'$, the inner sum of \eqref{MomentB} will be
	$$\sum_{i=1}^{s_k} r_{i,k}-\sum_{j=1}^{t_{k'}} r_{j,k'}=\sum_{r_n\in I_k}\alpha_nr_n-\sum_{r_n\in I_{k'}}\beta_nr_n$$ 
	and by rational independence \eqref{1}, if the sum vanishes, then $\alpha_n=\beta_n=0$. Thus, there is only the contribution when the inner sum doesn't vanish and, as above, this term decays as $R$ goes to infinity due to \eqref{jal}. Now, we can deduce the expression for the general case of \eqref{MomentB}. For the inner sum we can write:
	$$
	\sum_{k\in\cK^+}\left(\sum_{i=1}^{s_k} r_{i,k}-\sum_{j=1}^{t_k} r_{j,k}\right)=\sum_{k\in\cK^+}\sum_{r_n\in I_k}\gamma_n^k r_n,
	$$	
	so the integral is, by rational independence \eqref{1} and \eqref{5},
	\begin{align*}
	\Bint_{B(1)}e&\left(\sum_{k\in\cK^+}\left(\sum_{i=1}^{s_k} r_{i,k}-\sum_{j=1}^{t_k} r_{j,k}\right)\cdot Rx\right)dx =\\
	&\hspace{3cm}=\begin{cases}
	1 & \gamma_n^k=0 \hspace{1.5mm}\forall n,k,\\
	C_m\dfrac{J_{\Lambda+1}(2\pi R\norm{\sum_{k\in\cK^+}\sum_{r_n\in I_k}\gamma_n^k r_n})}{(R\norm{\sum_{k\in\cK^+}\sum_{r_n\in I_k}\gamma_n^k r_n})^{\Lambda+1}} & \text{otherwise.}
	\end{cases} 
	\end{align*}
This splits $\mathcal{C}$ into $\mathcal{C}_1\subset\mathcal{C}$ of all choices of $n_{i,k},n_{j,k}$ such that $\gamma^k_n=0$ and $\mathcal{C}_2\coloneqq \mathcal{C}\backslash\mathcal{C}_1$, as in \eqref{5.5}. Now,
$$
\sum_{\mathcal{C}_1} \prod_{k\in\cK^+}\prod_{i=1}^{s_k}\prod_{i=1}^{t_k} a_{i,k}\overline{a}_{j,k}= \prod_{k\in\cK^+}\sum_{\mathcal{C}_{k,1}}\prod_{i=1}^{s_k}\prod_{i=1}^{t_k} a_{i,k}\overline{a}_{j,k}=\prod_{k\in\cK^+}	\delta_{s_kt_k}\left(s_k!+O_{s_k}\left((\delta N)^{-1}\right)\right)
$$
where we used \eqref{MomentFactorial} for the last equality. Finally, the sum over $\mathcal{C}_2$, arguing as above, it is going to be $O_{D,N}\left(R^{-\Lambda-3/2}\right)$.
\end{proof}
Similarly, we can prove that the function $F_x$ has (asymptotically) real Gaussian moments for its $L^p$ norms:
\begin{prop}\label{prop:Lp mom}
	Let $p$ be a positive integer. Then, 
	$$\lim_{N\to \infty}\limsup_{R\to \infty}\Big|\Bint_{B(R)} |F_x(y)|^{2p}dx- \frac{(2p)!}{p! 2^p}\Big|=0,~\lim_{N\to \infty}\limsup_{R\to \infty}\Big|\Bint_{B(R)} |F_x(y)|^{2p+1}dx\Big|=0,$$	
	uniformly in $y\in B(W)$.
\end{prop}
\begin{proof}
	\begin{align*}
		\Bint_{B(R)}F_x^{p'}(y) dx=\frac{1}{(2N)^{p'/2}}\sum\prod_{i=1}^{p'}a_{n_i}e\left(\sum_{i=1}^{p'} r_{n_i}\cdot y\right)\Bint_{B(1)}e\left(Rx\cdot\sum_{i=1}^{p'}r_{n_i}\right)dx
	\end{align*}
where the sum $\Sigma$ means $\sum_{\{|n_i|\le N~:~1\le i\le {p'}\}}$. By \eqref{3.1.1}, the principal contribution will be when
$\sum_{i=1}^{p'} r_{n_i}=0$. In this case, if we define $\sum_{i=1}^{p'} r_{n_i}=\sum_{n=1}^N \alpha_n^+ r_n-\sum_{n=-N}^1 \alpha_n^- r_n=\sum_{n=1}^N \alpha_n r_n$, $\sum_{n=1}^N(\alpha_n^++\alpha_n^-)=2\sum_{n=1}^N\alpha_n^+=p'$ must be even, so for $p'=2p+1$ it will not be zero. For $p'=2p$ then, fixing a vector $\alpha^+\coloneqq(\alpha_n^+)_{n=1}^N$ there are $2p!$ ways of choosing $\{|n_i|\le N~:~1\le i\le 2p\}$ for that $\alpha^+$, so
	\begin{align*}
	\Bint_{B(R)}F_x^{2p} dx=\frac{1}{(2N)^{p}}(2p!)\sum 1+O_{N,p}(R^{-\Lambda-\frac32})
\end{align*}
by \eqref{1}, where the sum runs over all the possible $\alpha^+$. There are $\binom{N+p-1}{p}$ of those, so the leading term is:
\begin{equation*}
	\frac{1}{(2N)^{p}}(2p!)\frac{N+p-1!}{p!(N-1)!}=\frac{2p!}{2^pp!}\frac{N+p-1!}{(N-1)!N^p}=\frac{2p!}{2^pp!}(1+O_p(1/N)),
\end{equation*}
concluding the proof.
\end{proof}
\subsection{From deterministic to random: passage to Gaussian fields.}
The aim of this section is to prove the following technical proposition:
\begin{prop}
	\label{main prop pairing}
	Let $\phi_x$ be as in Section \textnormal{\ref{notation}}, $\varepsilon>0$, $W>1$ and $s\geq 0$ an integer.  
	Then there exist some  $K_0=K_0(\varepsilon,W,s)$,  $N_0=N_0(K,\varepsilon,W,s)$ and $R_0=R_0(N,\varepsilon,W,s)$ such that if  $K\geq K_0$, $\delta\lesssim K^{-m}$, $N\geq N_0$ and $R\geq R_0$ we have 
	$$
	d_P(\phi_x, F_{\mu}) < \varepsilon
	$$
where the convergence is with respect to the $C^s(B(W))$ topology. 
\end{prop} 
To ease the exposition, we divide the proof of Proposition \ref{main prop pairing} into two lemmas. In the first lemma we introduce the auxiliary field $F_{\mu_K}$ where
\begin{align}\label{DefmuK}
\mu_K\coloneqq \mu_{K,N}\coloneqq \sum_{k\in\cK}\gamma_{K,N}\delta_{\zeta^k}~\text{ where }~\gamma_{K,N}\coloneqq\frac{\mu_r(I_k)}{\sum_{k\in\cK}\mu_r(I_k)}\equiv \frac{\mu_r(I_k)}{\kappa_K^2} .
\end{align}
We note that, by Claim \ref{useful claim}, $\kappa_K\to 1$ as $\delta K^{m-1}\to0$.
\begin{lem}\label{PropConvMeasure1}
 Let $\varepsilon>0$, $s\geq 0$ and $K\geq 1$, $\delta>0$ be as in Section \textnormal{\ref{notation}}. Then there exist some  $N_0=N_0(\delta,K, \varepsilon,W,s)$ and $R_0=R_0(N,\varepsilon,W,s)$ such that for all $N\geq N_0$, $R\geq R_0$, we have 
$$
d_P(\kappa_K^{-1}\phi_x, F_{\mu_K}) < \varepsilon
$$
where the convergence is with respect to the $C^s(B(W))$ topology. 
\end{lem}

\begin{proof}
We begin by explicating the dependence of $R$ on $N$: we choose $N$ and $R$ such that the error term in Lemma \ref{moments} tends to zero for every $D$. In order to do so, we will follow a diagonal argument. Let us write explicitly the error terms in Lemma \ref{moments} as
\begin{equation*}
	|O_D(\delta^{-1}N^{-1})|\leq C_D\delta^{-1}N^{-1},\quad |O_{N,D}(R^{-\Lambda -3/2})|\leq C_{N,D}R^{-\Lambda -3/2},
\end{equation*}
and notice that, up to changing the constants, we may assume that $C_{D'}\le C_D$ and $C_{N,D'}\le C_{N,D}$ for any $D'<D$. Now, let us define $M_D$ and $R_{M,D}$ such that
\begin{equation*}
C_D M_D^{-1}\to 0\quad C_{M_D,D}R_{M_D,D}^{-\Lambda -3/2}\to 0
\end{equation*}
as $D\to\infty$. Then, we choose $N,R$ to go to infinity as any sequence satisfying $N_j\ge M_j$ and $R_j\ge R_{N_j,j}$. Taking said sequence of $N,R$, for any fixed $D$, we have
\begin{equation}\label{ErrorRN}
O_D(\delta^{-1}N_j^{-1})\to 0,  \quad O_{N_j,D}(R_j^{-\Lambda -3/2})\to 0
\end{equation}
as $j$ goes to infinity due to the fact that $C_{D'}\le C_D$, $C_{N,D'}\le C_{D,N}$ if $D'<D$. With this choice of $N$ and $R$ and mind, we simply say that $N,R$ tend to infinity.

 Let $\{b_k\}$ and $\{c_k\}$ be defined as in Section \ref{Gaussian moments}. Then,	since a Gaussian random variable is determined by its moments (as the moments generating functions exists) and the moments of all orders converge by \eqref{ErrorRN}, we can apply the method of moments,  \cite[Theorem 30.2]{Bi08}, to see that 
	\begin{align}\label{ConvPhi}
	\sum_{k}	\alpha_{k}b_{k}\overset{d}{\longrightarrow} \sum_{k}	\alpha_{k}c_{k} \quad\text{ as } R,N \rightarrow \infty
	\end{align}
  for any $\{\alpha_{k}\}\in \mathbb{R}^k$. Bearing in mind the definition of $\phi_x$ in \eqref{phi}, \eqref{eq:def muK}, the Cramér–Wold theorem \cite[Page 383]{Bi08}, implies that, for any $y_1,...,y_{\ell}\in B(W)$  with $\ell$ a positive integer and $\beta \geq 0$, we have 
\begin{align}\label{PointConv}
&\kappa_K^{-1}\left(\phi_x(y_1),...,\phi_x(y_{\ell})\right)\overset{d}{\longrightarrow} \left(F_{\mu_K}(y_1),...,F_{\mu_K}(y_{\ell})\right) & R,N \rightarrow \infty \nonumber \\
&\kappa_K^{-1}\left(D^{\alpha}\phi_x(y_1),...,D^{\alpha}\phi_x(y_{\ell})\right)\overset{d}{\longrightarrow} \left(D^{\alpha}F_{\mu_K}(y_1),...,D^{\beta}F_{\mu_K}(y_{\ell})\right) & R,N \rightarrow \infty 
\end{align}
where we have used the multi-index notation $D^{\alpha}\coloneqq \partial_{y_1}^{\alpha_1}...\partial_{y_n}^{\alpha_n}$. Thus, thanks to \eqref{PointConv} and  the discussion in Section \ref{SectWeakConv}, in order to prove the Lemma, we are left with checking the hypothesis of Proposition \ref{tightness}. 
 
 \textbf{Condition ii) in Proposition \ref{tightness}} 
Let $\phi^\alpha_x\coloneqq \kappa_K^{-1}D^\alpha_{y_i}\phi_x$ with $\alpha$ a multi-index, the Cauchy-Schwarz inequality gives 
\begin{align}                                      
|\phi^\alpha_x(y)-\phi^\alpha_x(y')|&\lesssim(2\pi)^{|\alpha|} |\sum_{k\in\cK} b_k(x)\mu_r(I_k)^{1/2}\prod_{i=1}^m(\zeta^k_i)^{\alpha_i}\left(e\left( \zeta^k\cdot  y\right)-e\left( \zeta^k\cdot y'\right)\right) | \nonumber \\
&\lesssim \norm{y-y'}\sum_{k\in\cK} |b_k(x)|. \label{5.6}
\end{align}
 Moreover, by \eqref{ConvPhi} and the Continuous Mapping Theorem  \cite[Theorem 2.7]{BI}, we have 
\begin{align}
\sum_{k\in\cK} |b_k(x)| \overset{d}{\longrightarrow} G\quad\text{ as } R,N\rightarrow \infty                                    
\end{align}
 where $G$ is a random variable with finite mean, i.e. the sum of folded normal variables. By Portmanteau Theorem and Chebyshev's inequality, we deduce that 
\begin{align}
\limsup_{R,N\to \infty}\vol_R\left(\sum_{k\in\cK} |b_k(x)|\ge \varepsilon' \right)\le \bP\left(G\ge \varepsilon' \right)\le \bE[G] \varepsilon'^{-1}, \label{5.7}
\end{align}
as $[\varepsilon',\infty)$ is a closed set.
Therefore, by \eqref{5.6} and \eqref{5.7}, using the notation in \eqref{continuity mod}, we have 
$$
\limsup_{R,N\to \infty}\vol_R\left(\omega_{\phi^\alpha_x}(\delta)\ge \varepsilon\right)\le\limsup_{R,N\to \infty}\vol_R\left(\sum_{k\in\cK} |b_k(x)|\ge C_{\cK}\varepsilon\delta^{-1}\right)\le  C'_{\cK}\varepsilon^{-1}\delta.
$$
Hence, we can conclude that, for all $\varepsilon>0$ and all $i$, we have 
$$\lim_{\delta\to 0}\limsup_{R,N\to \infty}\bP\left(\omega_{\phi^\alpha_x}(\delta)\ge \varepsilon\right)=0.$$
This establishes (ii) in Proposition \ref{tightness}.

\textbf{Condition i) in Proposition \ref{tightness}}  By \eqref{phi} and \eqref{UncorrelationCoeff}, for any point $y\in B(W) $ we have $\Bint_{B(R)}|\phi_x^0(y)|^2 dx =O(1)$ as in Proposition \ref{prop:Lp mom} for $p=1$. Thus, by the Chebyshev's inequality, we have  
$$
	\bP_R(|\phi^0_x(y)|>M)\lesssim M^{-2}.
$$
We can proceed similarly with $\phi_x^\alpha(y)$ and this establishes i).  
\end{proof}
To prove the next result, we need the following lemma, compare the statement with \cite[Lemma 4]{So} (here only a weaker version is needed).
\begin{lem}
	\label{LemConvGaussianFields}
	Let $\mu_n$ be a sequence of probability measures on $\S^{m-1}$ such that $\mu_n$ converges weakly to some probability measure $\mu$, then 
	 \begin{align}
	 \nonumber
	 	d_P\left(F_{\mu_n}, F_{\mu}\right)\rightarrow 0 \quad \text{ as }n \rightarrow \infty,
	 \end{align}
	 where the convergence is with respect to the $C^s(B(W))$ topology.
\end{lem}
 
We provide a proof of Lemma \ref{LemConvGaussianFields} in Appendix \ref{AppendixGaussianFields}. Using Lemma \ref{LemConvGaussianFields}, we deduce the following:
\begin{lem}
\label{PropConvMeasure2.1}
	Let $\varepsilon>0$ and $s\geq 0$. Then there exist some $K_0=K_0(\varepsilon ,W,s)$ and some $N_0=N_0(\varepsilon,W,s)$ such  that for all $K\geq K_0$, $\delta\lesssim K^{-m}$ and $N\geq N_0$, we have 
	$$
	d_P(F_{\mu_K}, F_{\mu}) < \varepsilon
	$$
	where the convergence is with respect to the $C^s(B(W))$ topology. 
\end{lem}
\begin{proof}

Let $\mu_{K,N}$ be as in \eqref{DefmuK} 
and let $\delta\lesssim K^{-m}$. Then, in light of Lemma \ref{LemConvGaussianFields} it is enough to prove the following: let $\varepsilon>0$, there exists some $K_0=K_0(\varepsilon)$ and some $N_0=N_0(\varepsilon)$ such  that for all $K\geq K_0$ and $N\geq N_0$, we have 
	\begin{align}
d_P (\mu_{K,N},\mu)  < \varepsilon. \nonumber
	\end{align}  
Since $\mu_r$ weak$^{\star}$-converges to $\mu$ as $N\rightarrow \infty$, we may assume that $N_0= N_0(\varepsilon)$ so that 
$d_P(\mu_r, \mu)< \varepsilon/2$. Therefore, using the triangle inequality, it is enough to prove that $d_P(\mu_{K,N},\mu_r)< \varepsilon/2$ for $K$ large enough depending on $\varepsilon$ only, which bearing in mind \eqref{def of d_p}, is equivalent to the following: 
\begin{align}
	&\mu_{K,N}(B)\leq \mu_r(B_{+\varepsilon/2}) + \varepsilon/2 &\mu_r(B) \leq \mu_{K,N}(B_{+\varepsilon/2}) +\varepsilon/2   \label{to prove}
\end{align} 
for all Borel sets $B\subset \S^{m-1}$. 	For the sake of simplicity, from now on we write $\mu_{K,N}=\mu_{K}$. By Claim \ref{useful claim} ii) for some $C>1$ if $r_n\in I_k$, then $\norm{r_n-\zeta^k}<CK^{-1}$, therefore 
\begin{align} \mu_{K}(B) = \frac{1}{2N}\sum_{\substack{ k\in \mathcal{K} \\ \zeta^{k}\in B}}\sum_{r_n \in I_k} 1 \Big/ \sum_{k\in\cK}\mu_r(I_k) \leq \frac{1}{2N}\sum_{r_n \in B_{+C/K}} 1 \Big/ \sum_{k\in\cK}\mu_r(I_k)  \nonumber
\end{align}
which, together with Claim \ref{useful claim} i) and our choice of $\delta\lesssim K^{-m}$, gives 
$$ \mu_{K}(B) \leq \mu_r(B_{+C/K}) + C/K.$$
This proves the first part of \eqref{to prove}, with $\varepsilon/2=C/K$. Since
$$ \mu_K(B_{+C/K})=\frac{\sum_{k\in \cK}\mu_r(I_k)\delta_{\zeta^k}(B_{+C/K})}{\sum_{k\in \cK}\mu_r(I_k)}\ge \frac{\sum_{k\in \cK}\mu_r(I_k\cap B)}{\sum_{k\in \cK}\mu_r(I_k)}=\frac{\mu_r(B)-\mu_r(B\cap I_\cK^c)}{\mu_r(I_\cK)},
$$
where $I_\cK\coloneqq\bigcup_{k\in\cK}I_k$. Therefore, we have 
$$  \mu_{r}(B) \leq \mu_K(B_{+C/K})+C/K . $$
This proves the second part of \eqref{to prove}, with $\varepsilon/2=C/K$ and hence Lemma \ref{PropConvMeasure2.1}.
\end{proof}
We are finally ready to prove Proposition \ref{main prop pairing}.
\begin{proof}[Proof of Proposition \textnormal{\ref{main prop pairing}}]
	The proposition follows from Lemma \ref{PropConvMeasure1} and Lemma \ref{PropConvMeasure2.1} together with the triangle inequality for the Prokhorov distance with the following order in the choice of the parameters: $W>1$, $s\geq 0$ a natural number, $\varepsilon>0$ are given, then $K$ is large depending on $\varepsilon,W,s$ according to Lemma \ref{PropConvMeasure2.1}, $\delta\lesssim K^{-m}$; $N$ is large depending on $\delta,\varepsilon,W,s$ according to  Lemma \ref{PropConvMeasure1} and Lemma \ref{PropConvMeasure2.1};  finally $R$ is large depending on all the previous parameters according to Lemma \ref{PropConvMeasure1}.
\end{proof}
\subsection{Concluding the proof of Theorem \ref{thm 1}}
We are finally ready to prove Theorem \ref{thm 1}: 
\begin{proof}[Proof of Theorem \textnormal{\ref{thm 1}}]
Fix $W>1$, $s\geq 0$ and $\varepsilon>0$. Let $K$ large enough according to Lemma \ref{PropConvMeasure2.1} applied with $\varepsilon/4$ and such that $CW^{2(s+l)+m}(\delta K^{m-1}+W^2K^{-2})<\varepsilon/4$ if $\delta\lesssim K^{-m}$ where $C,l$ were defined in Lemma \ref{first approx} and also such that $|\kappa_K^{-1}-1|$ is small enough by Claim \ref{useful claim}. Then we take an $N$ large enough according to Lemma \ref{PropConvMeasure1} applied with $\varepsilon/4$. Finally, let $R$ large enough as in Lemma \ref{PropConvMeasure1} applied with $\varepsilon/2$ and such that in Lemma \ref{first approx}, $O_N(R^{-\Lambda-3/2})<1$. With this, we have 
$$\Bint_{B(R)}\norm{F_x- \phi_x}^2_{C^s(B(W))} < \varepsilon/2	~\text{ and }~d_P(\phi_x, F_{\mu})< \varepsilon/2, $$
concluding the proof.

\end{proof}
\section{Proof of Proposition \ref{semi-locality m>2},  semi-locality.}
\label{last section}
Let $f$ be as in \eqref{function} and denote by $\mathcal{NI}(f,x,W)$ the number of nodal domains that intersect the boundary of $B(x,W)$. In order to prove Proposition \ref{semi-locality m>2} we need to obtain bounds on $\mathcal{NI}(f,x,W)$.  This is the content of the next section, where we follow the recent preprint of Chanillo, Logunov, Malinnikova and Mangoubi \cite{CLMM20}, see also Landis \cite{L63}.
\subsection{A bound on $\mathcal{NI}$}
 We  begin by introducing a piece of notation borrowed from \cite{CLMM20}:  we say that a domain $A$  is $c_0$-narrow (on scale $1$) if 
$$\frac{|A\cap B(x,1)|}{|B(x,1)|}\leq c_0$$ 
for all $x\in \overline{A}$. We will use the following bound [Section 3.2]\cite{CLMM20}:
\begin{lem}
	\label{bound intersections}
	Let $r\in (1,R)$ and denote by $\Omega$ a nodal domain of $f$. Then, we have  
	\begin{align} 
		\nonumber\mathcal{NI}(f, x , r)\lesssim_m \left|\{\Omega: \Omega \cap B(x,r) \textit{ is not $c_0$-narrow}\} \right|+ \mathfrak{N}_f(B(x,r))^{m-1} + r^{m-1},
\end{align}
where $\mathfrak{N}(\cdot)$ is as in Section \textnormal{\ref{sec:doubling index}}.
\end{lem} 
Since the proof of Lemma \ref{bound intersections} follows step by step \cite{CLMM20}, we decided to present it in Appendix \ref{copying}. We observe that, if a nodal domain is not $c_0$-narrow, then $|\Omega \cap B(x_0,1)|> c_1$ for some constant $c_1>0$ and for some $x_0\in\overline{\Omega}$. Thus, the number of \textit{non} $c_0$-narrow nodal domains in $B(x,W)$ is $O(W^m)$. Hence, Lemma \ref{Nazarov-Turan} together with Lemma \ref{bound intersections} gives the following bound: 
\begin{cor}
	\label{abs bound}
Let $F_x$ be as in \eqref{Fx}, then, provided that $N$ is sufficiently large with respect to $W$, we have  
$$	\nonumber\mathcal{NI}(f,x,W)\lesssim_m W^m + N^{m-1}\lesssim N^{m-1}.$$ 
\end{cor}

\subsection{Small values of $f$}
\label{growth of $f$}
 In this section we prove the following lemma which will also be our main tool in controlling the doubling index of  $f$.  
\begin{lem}
	\label{anti-concentration}
	Let $f$ be as in \eqref{function}, $\beta>0$ and $D>1$ be two parameters, then 
	\begin{align}
	\vol_R(|f(x)|\leq \beta)\lesssim \beta + N^{-D} + O_{N}(R^{-\Lambda-3/2}). \nonumber
	\end{align} 
\end{lem}
The proof relies on the following Hal\`{a}sz' anti-concentration inequality  \cite{H} and \cite[Lemma 6.2]{NV1}.
\begin{lem}[Halasz' bound]
	\label{Halas'z}
	Let $X$ be a real-valued random variable and let $\psi(t)=\mathbb{E}[\exp(itX)]$ be its characteristic function then there exists some absolute constant $C>0$ such that
	\begin{align}
	\mathbb{P}(|X|\leq 1)\leq C \int_{|t|\leq 1} |\psi(t)|dt. \nonumber 
	\end{align}
\end{lem}
We are now ready to prove Lemma \ref{anti-concentration}.
\begin{proof}[Proof of Lemma \textnormal{\ref{anti-concentration}}]
	Firstly we rewrite $f$ as 
	\begin{align}
	f(x)= \sum_{n=1}^Nb_n \cos (r_n \cdot x)+\sum_{n=1}^N c_n\sin (r_n \cdot x) \label{re-write f}
	\end{align}
	for some $b_n,c_n$ with $b_n^2+c_n^2=1/2N$. We apply Lemma \ref{Halas'z} to obtain 
	\begin{align}\label{A1}
	\vol_R(|f(x)|\leq \beta)\lesssim \beta \int_{|t|\leq 1/\beta} |\psi_{R}(t)|
	\end{align}
	where $\psi_{R}(t)=\psi_{f,R}(t)= \Bint_{B(R)}[ \exp(2\pi it f(x))]$. From now on, we may also assume that $1/\beta\geq 1$, as, if $1/\beta\leq 1$, then we can use the trivial bound $|\psi_{R}(t)|\leq 1$ on the right hand side of \eqref{A1} and conclude the proof. We need the following  Jacobi–Anger expansion \cite[page 355]{AS}:
	\begin{align}
	&e^{iz\cos\theta}= \sum_{\ell=-\infty}^{\infty}i^{\ell}J_{\ell}(z)e^{i\ell\theta} 
	&	e^{iz\sin\theta}= \sum_{l=-\infty}^{\infty}J_{\ell}(z)e^{i\ell\theta}. \nonumber
	\end{align}

	Then, by \eqref{re-write f}, we have 
	\begin{align}
	\exp( i t f(x))&= \prod_{|n|\le N} \exp(  t( b_n\cos (r_n \cdot x)+ c_n \sin(r_n \cdot x)) \nonumber \\
	&= \prod_{|n|\le N}\left[ \sum_{\ell=-\infty}^{\infty}i^\ell J_{\ell}( t b_n)e^{i\ell(r_n \cdot x)}\right] \cdot \left[ \sum_{\ell=-\infty}^{\infty}J_{\ell'}(t c_n)e^{i{\ell'}(r_n \cdot x)} \right] \nonumber \\
	&=\sum_{\substack{\ell_1,...\ell_N\\ \ell'_1,...,\ell'_N}} \left(\prod_{|n|\leq N}i^{\ell_n}J_{\ell_n}(t b_n) J_{\ell'_n}(t c_n) \right) \cdot  e^{(i r_1(\ell_1+\ell'_1)+...+r_N(\ell_N+\ell'_N))\cdot x} . \label{14}
	\end{align}
	Thanks to the rapid decay of Bessel functions as the index $\nu\to\infty$ for a fixed argument $z$, that is
	$$
	J_{\nu}\left(z\right)\lesssim\frac{1}{\sqrt{2\pi\nu}}\left(\frac{ez}{2\nu}\right)^{\nu},
	$$
	and bearing in mind \eqref{1} and \eqref{5}, we can integrate \eqref{14} with respect to $x$ which, using  Fubini, gives
	\begin{align}\label{6.1}
	\psi_{R}(t/(2\pi))= \sum_{\ell_1,...,\ell_N=-\infty}^{\infty}\prod_{|n|\le N} i^{\ell_n} J_{\ell_n}( tb_n) J_{-\ell_n}( tc_n) + O_{N}\left( R^{-\Lambda-3/2} \right) 
	\end{align}
	For the first term on the RHS of \eqref{6.1}, we rewrite it as 
	\begin{align}\label{6.2}
	\sum_{\ell_1,...,\ell_N=-\infty}^{\infty}\prod_{|n|\le N} i^{\ell_n} J_{\ell_n}( tb_n) J_{-\ell_n}( tc_n)= \prod_{|n|\le N}\left(\sum_{\ell=-\infty}^{\infty}(-i)^{\ell}J_{\ell_n}( tb_n) J_{\ell_n}( tc_n)\right) 
	\end{align}
	where we have used the fact that $J_{-\ell}(x)=(-1)^{\ell}J_{\ell}(x)$. By Graf’s addition theorem \cite[page 361]{Wbook}, we have 
	\begin{align}\label{6.5}
	J_0(\sqrt{x^2+y^2})= \sum_{\ell=-\infty}^{\infty} J_{\ell}(x)J_{\ell}(y)\cos\left(\frac{\pi}{2}l\right),\quad \sum_{\ell=-\infty}^{\infty}J_{\ell}(x)J_{\ell}(y)\sin\left(\frac{\pi}{2}l\right)=0. 
	\end{align}  
	Writing $(-i)^l= \cos(\pi \ell/2)- i \sin(\pi \ell/2)$ and applying \eqref{6.5}, bearing in mind that $b_n^2+c_n^2=1/2N$, we obtain 
	\begin{align} \label{6.6}
	\sum_{\ell=-\infty}^{\infty}(-i)^{\ell}J_{\ell_n}( tb_n) J_{\ell_n}( tc_n)= J_0\left(\frac{\sqrt{2}t}{\sqrt{N}}\right). 
	\end{align}
	Finally, inserting \eqref{6.6} into \eqref{6.2}, we deduce that 
	\begin{align} \label{6.9}
	\psi_{R}(t/(2\pi))=  \left(J_0\left(\frac{\sqrt{2}t}{\sqrt{N}}\right)\right)^N +O_{N}\left( R^{-\Lambda-3/2} \right) .
	\end{align}
	Let us denote the first summand by $\Psi_{R}^N(t)$. 
	By the very definition of Bessel functions \cite[Page 375]{AS}, we have 
	\begin{align}\nonumber
	J_{0}(x)= \sum_{q=0}^{\infty}\frac{(-1)^q}{q!\Gamma(q+1)}\left(\frac{x}{2}\right)^{2q} 
	\end{align} 
	Therefore, $J_0(x)= 1- \Gamma(2)^{-1}x^{2} + O(x^4)\leq e^{-cx^2}$ for some $c>0$ and  $x$ sufficiently small. Thus, bearing in mind \eqref{6.9}, we have  
	\begin{align}\label{6.3}
	\Psi_{R}^N(t)\lesssim e^{-Ct^2} 
	\end{align}
	for all $t\leq c_1N^{1/2}$ for some sufficiently small constant $c_1>0$. For $t\geq c_1 N^{1/2}$, we use that fact that $|J_0(x)|\leq\alpha<1$ for $x>c_1$ and for some $0<\alpha<1$, to obtain the bound 
	$$\Psi_{R}^N(t)\lesssim \alpha^N \lesssim N^{-D} $$ 
for any $D\geq 1$. Thus,
$$
\beta \int_{|t|\leq 1/\beta}|\Psi_{R}^N(t)| dt\lesssim \beta \int_{\R}e^{-Ct^2}dt+\beta\int_{-1/\beta}^{1/\beta} N^{-D}dt + O_{N}(R^{-\Lambda-3/2})
$$
obtaining the desired result.
\end{proof} 

As a consequence of Lemma \ref{anti-concentration}, we deduce the following: 

\begin{lem}
	\label{anti-concentration v2}
	Let $f$ be as in \eqref{function}, $D,H>1$ be some (arbitrary but fixed) parameters. Then,  we have 
	\begin{align}
	\vol_R\left(\mathfrak{N}_f( B(x,H))>Q\right)\lesssim_D \begin{cases}
	1 & Q\leq C' H \\
	\frac{1}{Q^{D}} + \frac{Q^{2D}}{e^{Q}} & Q>C' H
	\end{cases} + O_{N,H}\left(R^{-\Lambda-3/2}\right)  \nonumber
	\end{align}
	uniformly for all $Q\leq N$, $x\in B(R)$ and some absolute constant $C'>0$. 
\end{lem}
\begin{proof}
	 Now, consider $h(\cdot,t)\coloneqq f(\cdot) \cdot e^{2\pi t}$ on $\tilde{B}(H)=\tilde{B}(x,H)=B(x,H)\times [-H/2,H/2]$, then $\sup_{ B(x,H)}|f|\leq \sup_{\tilde{B}(H)}|h|$. Write $h_H(\cdot)=h(H\cdot)$ and $f_H(\cdot)=f(H\cdot)$, using elliptic regularity \cite[p.332]{Ebook}, we have 
	$$\sup_{{B}(H)}|f|e^{\pi H}=\sup_{\tilde{B}(H)}|h|= \sup_{\tilde{B}(1)}|h_H|\lesssim  ||h_H||_{L^2(\tilde{B}(2))}\lesssim e^{2\pi H} ||f_H||_{L^2(B(2))},$$
	where the constants in the notation are independent of $H$. Thus, letting $H'=\varkappa_m H$ $B(2)= B'$, we obtain
	\begin{align}
		\label{4.2.1}
\mathfrak{N}_f(B(x,H))\leq	\log \frac{ \sup_{ B(x,1)} |f_{H'}|}{|f(x)|}+1 \le C_1+\pi H + \log \frac{||f_{H'}||_{L^2(B')}}{|f(x)|}  . \end{align}
 Thanks to \eqref{4.2.1}, we have for $B''=B(2\varkappa_m)$ 
 $$ \vol_R( \mathfrak{N}_f(x,H)>Q)\leq  \vol_R\left( \log\frac{||f_H||_{L^2(B'')}}{|f(x)|}>Q - cH \right),$$
 for some absolute constant $c>0$. Since $Q-cH> Q/2$ for $Q> C'H$ for $C'=2c>0$, we obtain that, in order to prove the lemma, it is enough to prove the following: 
	\begin{align}	\vol_R\left(\log	 \frac{||f_H||_{L^2(B'')}}{|f(x)|}>Q/2\right) \lesssim_D 1/Q^{D} + \frac{Q^{2D}}{e^{Q}} + O_{N,H}\left(R^{-\Lambda-3/2}\right), \label{second claim} 
	\end{align}
	for $Q\geq C'H$. Let $I(x)\coloneqq	\log	 (||f_H||_{L^2(B'')}/|f(x)|)$ and $\beta>0$ be some parameter to be chosen later, then 
	\begin{align}
	\vol_R\left(	 I>Q/2\right) = 	\vol_R\left( I>Q/2 \hspace{1mm} \text{and} \hspace{1mm} |f(x)|<\beta\right) + \vol_R\left(	 I>Q/2 \hspace{1mm} \text{and} \hspace{1mm} |f(x)|\geq \beta \right) \label{5.2.1}
	\end{align}
	First, we bound the first term on the RHS of \eqref{5.2.1}. By Lemma \ref{anti-concentration}, we have 
	\begin{align}
	\vol_R(	I \geq Q /2\hspace{3mm}\text{and} \hspace{3mm} |f(x)|<\beta ) \leq \vol_R( |f(x)|\leq \beta )\lesssim \beta + O_{N}\left(R^{-\Lambda-3/2}\right) \label{12}
	\end{align}
	provided $\beta \geq N^{-D}$. For the second term on the RHS of \eqref{5.2.1}, we notice that
	\begin{align}
	\vol_R(	I \geq Q/2\text{ and } |f(x)| \geq \beta) \leq 	\vol_R\left( \norm{f}_{L^2(B'')}> \beta e^{Q/2}\right). \label{13}
	\end{align}
	However, bearing in mind that $B''=B(x,2\varkappa_m)$, using \eqref{1}, \eqref{5} and Fubini, we have 
	$$\Bint_{B(R)}||f_H||_{L^2(B'')}^2=\vol B'' + O_{N}\left((HR)^{-\Lambda-3/2}\right).
	$$
Thus, Chebyshev's inequality gives
	\begin{align}                      
	\vol_R\left( \norm{f_H}_{L^2(B'')}> \beta e^{Q/2}\right) \lesssim \frac{1}{\beta^2 e^{Q}} + O_{N}\left((HR)^{-\Lambda-3/2}\right) \label{5.2.2}.
	\end{align}
	Hence, putting \eqref{5.2.1}, \eqref{12}, \eqref{13} and \eqref{5.2.2} together, we get 
	\begin{align}
	\vol_R(	I > Q)\lesssim  \beta + \frac{1}{\beta^2 e^{Q}} + O_{N,H}\left(R^{-\Lambda-3/2}\right) \nonumber 
	\end{align}
	finally, we take $\beta=1/Q^{D}$, so the condition $\beta \geq N^{-D}$ is equivalent to $N\geq Q$ and conclude the proof of \eqref{second claim}. 
\end{proof}

\subsection{Proof of Proposition \ref{semi-locality m>2}}
We are finally ready to prove Proposition \ref{semi-locality m>2}
\begin{proof}[Proof of Proposition \textnormal{\ref{semi-locality m>2}}]
 For short we write $\mathcal{	NI}(f,x,W)= \mathcal{	NI}(F_x,W)$. By \cite[Lemma 1]{NS} for $r=W$, we have 
\begin{align*}
	\frac{1}{\vol B(R)}\int_{B(R-W)}&\frac{\mathcal{N}(F_x,W)}{\vol B(W)}dx\le\frac{\mathcal{N}(f,R)}{\vol B(R)}\le\\
	& \le\frac{1}{\vol B(R)}\left(\int_{B(R+W)}\frac{\mathcal{N}(F_x,W)}{\vol B(W)}dx+ \int_{B(R+W)}\frac{\mathcal{NI}(F_x,W)}{\vol B(W)}dx\right).
\end{align*}
By Faber-Krahn inequality,
\begin{equation*}
	\int_{B(R+W)}\frac{\mathcal{N}(F_x,W)}{\vol B(W)}dx-\int_{B(R)}\frac{\mathcal{N}(F_x,W)}{\vol B(W)}dx\lesssim \vol B(R)\frac{\left(R+W\right)^m-R^m}{R^m},
\end{equation*}
which is $O(\vol B(R)W/R)$ by the binomial theorem, similarly for $B(R)$ and $B(R-W)$. Thus, we have 
	\begin{align}\nonumber
	\frac{\mathcal{N}(f,R)}{\vol B(R)}= \frac{1}{\vol B(W)}\Bint_{B(R)}&\mathcal{N}(F_x,W)dx\nonumber\\
	&+ O\left( \frac{1}{W^m}\Bint_{B(R+W)} \mathcal{NI}(F_x,W) dx\right)+O\left(\frac{W}{R}\right) .
\end{align}
	Therefore, it is enough to prove the following:
	\begin{equation}\label{eq:bound NI}
	\frac{1}{W^m}\Bint_{B(R+W)} \mathcal{NI}(F_x,W)dx\lesssim\frac{1}{W}\left(1 + O_{N,W}\left( R^{-\Lambda-3/2}\right)\right).
	\end{equation}
First, we observe that if we  cover $\partial B(x,W)$ with $O(W^{m-1})$ ($m$-dimensional) balls of radius $100$ with centres $x+y_i$, then 
\begin{align} \mathcal{NI}(F_x,W)= \mathcal{NI}(f,x,W)&\leq \sum_i \left( \mathcal{NI}(f,x+y_i,100) + \mathcal{N}(f,x+y_i,100)\right) \nonumber  \\
	& \leq  \sum_i  \mathcal{NI}(f,x+y_i,100) + O\left( W^{m-1}\right), \nonumber
	\end{align}
where in the second inequality we have used the Faber-Krahn inequality. Therefore, thanks to Lemma \ref{bound intersections} applied with $r=100$,  we have 
\begin{align}
	&\frac{1}{W^m}\Bint_{B(R+W)}\mathcal{NI}(F_x,W)dx  \lesssim\frac{1}{W^m} \sum_i \Bint_{B(R+W)} \mathfrak{N}_f(B(x+y_i,100))^{m-1}dx + \nonumber \\
	&+\frac{1}{W^m} \sum_{i}\Bint_{B(R+W)} \left|\{\Omega: \Omega \cap  B(x+y_i,100) \textit{ is not $c_0$-narrow}\} \right|dx +O\left(\frac{1}{W}\right) \label{4.1.1}
\end{align} 
If $\Omega \cap  B(x+y_i,100)$ is not $c_0$-narrow, then  
$|\Omega \cap  B(x+y_i,100)|> c_1$. Thus 
$$\left|\{\Omega: \Omega \cap  B(x+y_i,100) \text{ is not $c_0$-narrow}\}\right|=O(1),$$
bearing in mind that the sum over $i$ has $O(W^{m-1})$ terms, the second term on the right hand side of \eqref{4.1.1} is $O(W^{-1})$.

Thus, it is enough to bound the first term to the right hand side of $\eqref{4.1.1}$. Thanks to Lemma  \ref{Nazarov-Turan},
 $ \mathfrak{N}_f(B(x,100)) \lesssim N$. Thus, we  have 
\begin{align}\Bint_{B(R+W)} \mathfrak{N}_f(B(x+y_i,100))^{m-1}dx &= \int_{1}^{(CN)^{m-1}} \vol_R ( \mathfrak{N}_f(B(x+y_i,100))^{m-1} > t ) dt \nonumber \\ &\lesssim 1+ O_{N,100} \left( (R+W)^{-\Lambda-3/2}\right),  \nonumber
	\end{align}
 where in the second inequality we have used Lemma \ref{anti-concentration v2}, with $D=m$  and $H=100$. This concludes the proof of the proposition.  
\end{proof}

\section{Proof of Theorem \ref{thm 5}}
\label{section 5}
\subsection{Convergence in mean}
\label{convergence in mean}
The aim of this section is to show how Theorem \ref{thm 1} implies convergence in mean of $\mathcal{N}(\cdot)$. That is, we prove the following proposition: 
\begin{prop}\label{PropConvE2}Let $W\geq 1$ and $\mathcal{S}\subset H(m-1)$. Then we have 

	\begin{align*} 
\lim_{N\to\infty}\limsup_{R\to \infty}	\bigg|\Bint_{B(R)}\cN(F_x,\mathcal{S},W) -\bE[\cN(F_{\mu},\mathcal{S},W)]\bigg|=0 .
	\end{align*}
	Moreover, the conclusion also holds for $\cN(\cdot,T)$, as in Theorem \textnormal{\ref{thm 5}}.
\end{prop}
To ease the exposition we split the proof of Proposition \ref{PropConvE2} into a series of preliminary results.
\subsection{Continuity of $\mathcal{N}(\cdot)$}
In this section we show that $\mathcal{N}(\cdot,W)$, $\cN(\cdot,[\Sigma],W)$ and $\cN(\cdot,T,W)$ are continuous functionals on a particular subspace of $C^1$ functions. This is a consequence of Thom's Isotopy Theorem \ref{ThThom} and it refines the estimates of ``Shell Lemma'' in \cite{NS}. In order to state the main result of this section we need to introduce some notation:
\begin{equation*}
\snabla g:= \nabla g- \frac{x\cdot \nabla g}{|x|^2}x,
\end{equation*}
that is, the ``spherical'' part of the gradient. Also,
$\Psi_g\coloneqq|g|+\norm{\nabla g}$,
 $\slashed{\Psi}_g\coloneqq|g|+\norm{\snabla g}$ and for $W>1$, let us define
\begin{align} 
	C^1_*(W)\coloneqq\left\{g\in C^1(B(W+w))~|~\Psi_g>0 \text{ on }{B(W+w)}\text{ and }\slashed{\Psi}_g>0 \text{ on }\partial{B(W)}\right\}.\label{def C star}
	\end{align}
The parameter $w>0$ could be as small as we want. For the sake of simplicity, hereafter we assume $w=1$. We then prove the following,
\begin{prop}\label{PropContBdd}Let $W>1$ be fixed, $\mathcal{S}\in H(m-1)$ and $T\in \cT$ a finite tree.  Then $\cN(\cdot,\mathcal{S},W)$ and $\cN(\cdot, T,W )$ are continuous functionals on $C^1_*(W)$.
\end{prop}
Before starting the proof, we observe that  the condition on $\slashed{\Psi} $ is used to rule out the possibility  that the nodal set touches the boundary of the ball tangentially at one point. 

\begin{proof}[Proof of Proposition \textnormal{\ref{PropContBdd}}]
	Let $V\coloneqq B(W+1)$, $V_1\coloneqq B(W+1/2)$, $V_0\coloneqq B(W)$ and $h\in C^1_*(W)$, as $\Psi_h>0$, there is a finite number of connected components in $V_0$, i.e., 	 $h^{-1}(0)\cap V_0$ has components $\{\Sigma_i\}_{i=1}^I\sqcup \{\Sigma^*_j\}_{i=1}^J$,
	where $\Sigma_i\subset V_0$ for all $i\in\{1,\ldots, I\}$  and $\Sigma_j^*\cap \partial V_0\neq \emptyset$  for all $j\in\{1,\ldots, J\}$. To treat the $0$-level set as a boundaryless manifold, let $\chi$ be a smooth radial step function which is zero in $\overline{V}_1$ and greater than $\sup_{V}|h|$ on the boundary of $V$, say $T$. Moreover, we observe that the condition $\slashed{\Psi}_h>0$ implies that $\Sigma_j^*\cap \partial V_0$ is not a point: for $m=2$ this follows from the definition of $\slashed{\Psi}_h>0$;  for $m>2$ it follows from the fact that, as the intersection is transversal, it must be a submanifold on the boundary of codimension 1.  
	
	Therefore, it is possible to define $d_j>0$ as the maximal	 distance between some 
	\begin{equation}\label{eq:def x}
		x\in\left(\Sigma_j^*\backslash {V_0}\right) \cap\overline{V_1} 
	\end{equation}
and $\partial V_0$, i.e., $d_j=\max_x\text{dist}\left(x\in\left(\Sigma_j^*\backslash {V_0}\right)\bigcap\overline{V_1},\partial V_0\right)=:\text{dist}\left(x_j,\left(\Sigma_j^*\backslash {V_0}\right)\bigcap\overline{V_1}\right)$ for some $x_j\in\left(\Sigma_j^*\backslash {V_0}\right) \cap\overline{V_1}$. Now, we are going to apply Thom's Theorem \ref{ThThom} to $h+\chi$ on $V$. Note that $h+\chi$ has the same nodal set as $h$ on $\overline{V_1}$ and define $\Gamma_j^*$ the connected component (boundaryless) of the nodal set of $h+\chi$ which equals $\Sigma_j^*$ on $\overline{V_1}$. Let $U_i\subset V_0$ be the open neighbourhood of $\Sigma_i$ and similarly $U_j^*\subset V$ of $\Gamma_j^*$, both given by the theorem. Let us also take $\varepsilon_i>0$ and $d_j>\varepsilon_j^*>0$. By Theorem \ref{ThThom}, there is $\delta_i,\delta_j^*>0$ such that if $g'$ satisfies 
	$$
	\| h- g'\|_{C^1(U_i)}<\delta_i,\quad 	\| h+\chi- g'\|_{C^1(U_j^*)}<\delta_j^*
	$$
	then $g'$ has a nodal component in $U_i$, $U_j^*$ diffeomorphic to $\Sigma_i$, $\Sigma_j^*$ (respectively) and the diffeomorphism satisfies
	$$
	\|\Phi_j^*-\textnormal{id}\|_{C^1(\R^m)}<\varepsilon_j^*.
	$$
	If we define $\delta\coloneqq\min_{i,j}\{\delta_i,\delta_j^*,(T-\sup_V |h|)/2\}$ with $\norm{g-h}_{C^1(V)}<\delta$, and $g'=g+\chi$, then the connected component of $g'$ diffeomorphic to $\Gamma_j^*$ cannot lie inside $V_0$. Indeed,  if $x_j$ is defined as in \eqref{eq:def x}, $\norm{\Phi_j^*(x_j)-x_j}<d_j$, so $\Phi_j^*(x_j )$ is outside $V_0$. Finally letting $U\coloneqq \left(\left(\cup_i U_i\right)\cup\left(\cup_j U_j^*\right)\right)\cap V_0$, if 
	\begin{align}\label{eq:delta cond}
	\norm{g-h}_{C^1(V)}<\min\{\delta,\min_{x\in \overline{V_0}\backslash U} \{h\}>0\},
	\end{align}
	then $g$ satisfies the hypotheses of Thom's Theorem \ref{ThThom} for all the components and it cannot vanish outside $U$. Therefore
	$$
	\cN(h,[\Sigma],W)=\cN(g,[\Sigma],W)\quad\forall~\Sigma\in H(m-1),
	$$
	in particular, $\cN(h,\mathcal{S},W)=\cN(g,\mathcal{S},W).$
	The proof of $\cN(h, T, W)$ is similar as $\Phi_i$ of Theorem \ref{ThThom} is the identity outside a proper subset of $U_i$. 
\end{proof}

\begin{claim}
	\label{open} 
	With the notation of Proposition \textnormal{\ref{PropContBdd}},	$C^1_*(W)\subset C^1(B(W+1))$ is an open set. 
\end{claim}
\begin{proof}
	If $h_n\to h$ in the $C^1$ topology and $|h_n|(y_n)+\norm{\nabla h_n}(y_n)=0$, we can choose $y_{n_j}\to y$ so
	$$
	|h(y)-h_{n_j}(y_n)|\le |h(y)-h(y_{n_j})|+|h(y_{n_j})-h_{n_j}(y_{n_j})|,
	$$
	which goes to zero as $j\to\infty$, as the convergence is uniform, and similarly for the gradient and $\snabla$. Hence, the complement of $C^1_*(W)$ is closed.	
\end{proof}

\subsection{Checking the assumptions}
In this section, we give a sufficient condition on $\nu$ for the Gaussian field $F_{\nu}$ to belong to $C^1_{*}(W)$ with the notation of Proposition \ref{PropContBdd}. As our paths are a.s. analytic, we have the following lemma, see also \cite[Lemma 6]{NS}.
\begin{lem}[Bulinskaya's lemma]
	\label{Bulinskaya's Lemma}
	Let $F=F_{\nu}$, with $\nu$ an Hermitian measure supported on the sphere and $s\geq 1$. If $\nu$ is not supported on a hyperplane, then $F\in C^1_*(W)$ almost surely, where is $C^1_*(W)$ as in \textnormal{\eqref{def C star}}. 
\end{lem}
\begin{proof}
	The proof of $\Psi>0$ is a straightforward application of \cite[Proposition 6.12]{AW09} as the density of $(F,\nabla F)(x)$ is independent of $x\in B(W+1)$. Indeed, $$\bE\left(\partial_i F(x) \partial_j F(x)\right)=4\pi^2\int_{\S^{m-1}}\lambda_i\lambda_jd\nu(\lambda)=:\bar{\Sigma}_{ij}$$ for $i=0,\ldots,m$ where $\partial_0\coloneqq\textnormal{id}$ and $\lambda_0\coloneqq 1$. Note that as $\nu$ is Hermitian, $\Sigma_{i0}=0$ for $i>0$, that is, $F(x)$ and $\nabla F(x)$ are independent. If $\Sigma=(\bar{\Sigma}_{ij})_{i,j=1}^m$, then $	\det\Sigma= 0$ is equivalent to the existence of some $u\in \mathbb{R}^m$ such that
	\begin{equation}\label{non-degeneracy}
		\int_{\S^{m-1}}(\lambda\cdot u)^2d\nu(\lambda)=0.
	\end{equation}
However, this is not possible since $\nu$ is not supported on a hyperplane, thus  $	\det\Sigma \neq 0$ and $\Psi>0$.

 For $\slashed{\Psi}>0$, consider a local parametrization $(U,\varphi)$ of $S$ with basis $\{e_i\}_{i=1}^{m-1}$ of the tangent space $T_x S$ where $x\in S\coloneqq\partial B(W)$. Then, 
	$$\Var \snabla F(x)=E(x)\cdot \Sigma\cdot E(x)^t$$
	where 
	$E^t\coloneqq\left(\begin{array}{c}{e}_{1},\ldots, {e}_{m-1}\end{array}\right).
	$
	Thus, the variance is not positive definite if and only if there exists $v\in\R^{m-1}$ non-zero such that $u\coloneqq\sum_{i=1}^{m-1}v_i e_i$ and $u\cdot \Sigma\cdot u^t=(4\pi^2)\int_{\S^{m-1}}(\lambda\cdot u)^2d\nu(\lambda)=0$, in contradiction (again) with the fact that $\nu$ is not supported on a hyperplane. Thus, $\det \left(\Var \snabla F\lvert_{\varphi(U)}\right)>\delta>0$, so by Bulinskaya applied to $Y\coloneqq(F,\snabla F)$ we conclude $\slashed{\Psi}\lvert_{\varphi(U)}>0$, then proceed analogously with the other local parametrizations of the (finite) atlas.
\end{proof}
As a consequence of Proposition \ref{PropContBdd} and Lemma \ref{Bulinskaya's Lemma}, we have the following: 
\begin{lem}
	\label{middle step}
	Let $\varepsilon>0$, $W>1$ and $\mathcal{S}\subset H(m-1)$. Then there exist some $K_0=K_0(\varepsilon,W)$, $N_0=N_0(\varepsilon,W)$ such that for all $k\geq K_0$, $N\geq N_0$ and $\delta \lesssim K^{-m+1}$ we have 
	\begin{align}
	\left|\mathbb{E}[\mathcal{N}(F_{\mu_{K,N}},\mathcal{S},W)]-\mathbb{E}[\mathcal{N}(F_{\mu},\mathcal{S},W)]\right|\leq \varepsilon, \nonumber
	\end{align}
where $\mu_{K,N}$ is as in \textnormal{\eqref{DefmuK}}. Moreover, the conclusion also holds for $\cN(\cdot,T)$, as in Theorem \textnormal{\ref{thm 5}}.
\end{lem}
\begin{proof}
	Thanks to Lemma \ref{Bulinskaya's Lemma}, $F_{\mu_{K,N}}\in C_{*}^1(W)$ a.s., thus the lemma follows directly from Lemma \ref{PropConvMeasure2.1} with Portmanteau Theorem. We can apply it in light of the fact that   $\mathcal{N}(F_{\mu_{K,N}},\mathcal{S},W)\le\mathcal{N}(F_{\mu_{K,N}},W)=O(W^m)$ uniformly for all $K$ and $N$ by the Faber-Krahn inequality and Proposition \ref{PropContBdd} which ensures that $\mathcal{N}(F_{\mu_{K,N}},\mathcal{S},W)$ is a continuous functional on $C_{*}^1(W)$. 
\end{proof}
\begin{lem}
\label{middle step2}
	Let $\varepsilon>0$, $W>1$ and $\mathcal{S}\subset H(m-1)$. Then there exist some  $K_0=K_0(\varepsilon,W)$, $N_0=N_0(K,\varepsilon,W)$, $R_0=R_0(N,\varepsilon,W)$ such that for all $K\geq K_0$, $\delta\lesssim K^{-m}$,  $N\geq N_0$ and  $R\geq R_0$, we have 
	\begin{align*} 
		&\bigg|\Bint_{B(R)}\cN(\phi_x,\mathcal{S},W)dx-\bE[\cN(F_{\mu},\mathcal{S},W)]\bigg|< \varepsilon,
	\end{align*}
 Moreover, the conclusion also holds for $\cN(\cdot,T)$, as in Theorem \textnormal{\ref{thm 5}}.
\end{lem}
\begin{proof}
	 By Lemma \ref{middle step}, it is enough to prove that, under the assumptions of Lemma \ref{middle step2}, we have   
	\begin{align}
		\label{5.3.1}
		\left|\Bint_{B(R)}\mathcal{N}(\phi_x,\mathcal{S},W)dx - \mathbb{E}[\mathcal{N}(F_{\mu_{K}},\mathcal{S},W)]\right|\leq \varepsilon/2
	\end{align}	
	First, since $C^1_*(W)$ is open by Claim \ref{open} and  as $\bP\left(F_{\mu_K} \in C^1_*(W)^c\right)=0$ by Lemma \ref{Bulinskaya's Lemma},   Portmanteau Theorem jointly with Lemma \ref{PropConvMeasure1} gives
	\begin{align}\label{ProbC*}
		&\vol_R\left(\phi_x\in C^1_*(W)\right)\to 1
	\end{align}
	as $R,N$ go to infinity according to Lemma \ref{PropConvMeasure1} depending on $K\geq 1$ (and thus $\delta>0$). Thus,	by the Continuous Mapping Theorem and Lemma \ref{PropConvMeasure1} for $W+1$:
	$$\mathcal{N}(\phi_x,\mathcal{S},W) \overset{d}{\rightarrow} \mathcal{N}(F_{\mu_{K}},\mathcal{S},W).$$ 
	Therefore Lemma \ref{th:DCTh} (using Faber-Krahn inequality) implies the desired result as long as $R,N$ go to infinity as in  Lemma \ref{PropConvMeasure1} depending  on $K\geq 1$ (and thus $\delta>0$).
\end{proof}
\subsection{Proof of Proposition \ref{PropConvE2}}
We are finally ready to prove Proposition \ref{PropConvE2}. 

\begin{proof}[Proof of Proposition \textnormal{\ref{PropConvE2}}]	
In light of Lemma \ref{middle step} and \ref{middle step2} it is enough to prove that, given $\varepsilon>0$, there exists some  $K_0=K_0(\varepsilon,W)$, $N_0=N_0(K,\varepsilon,W)$, $R_0=R_0(N,\varepsilon,W)$ such that for all $K\geq K_0$, $\delta\lesssim K^{-m}$,  $N\geq N_0$ and  $R\geq R_0$, we have 
\begin{align*} 
 &\bigg|\Bint_{B(R)}(\cN(F_x,\mathcal{S},W)- \cN(\phi_x,\mathcal{S},W))dx\bigg|< \varepsilon .
\end{align*}
	First, by Lemma \ref{first approx}, we have 
	\begin{align}
	\Bint_{B(R)}\norm{F_x- \phi_x}^2_{C^s(B(W'))} dx\le C W'^{2(s+l)+m}\left(\delta K^{m-1} + W'^2K^{-2}\right)\left(1+O_N(R^{-\Lambda-3/2})\right) \nonumber,
	\end{align}
	where $W'\coloneqq W+1$. Let \begin{align}\epsilon\equiv\epsilon(K,N,W')\coloneqq C  W'^{2(s+l)+m}\left(\delta K^{m-1} + W'^2K^{-2}\right)\left(1+O_N(R^{-\Lambda-3/2})\right). \label{epsilon}
		\end{align}Let us denote $\phi_x$ by $\phi_x^{\epsilon}$ and similarly for $F_x$, let $B(R)^*$ be the set where $\phi_x\in C^1_*(W)$. We claim that
	\begin{align}
	\Bint_{B(R)^*} \left(\cN(F_x^\epsilon,\mathcal{S},W))- \cN(\phi_x^\epsilon,\mathcal{S},W))\right) dx \to 0 \text{ as }\epsilon\rightarrow 0. \label{6.3.1}
	\end{align}
	For the sake of contradiction, let us suppose that there exist some $\gamma>0$ and a sequence $\epsilon_n\to 0$, such that
	$$
	|\Bint_{B(R)^*} \left(\cN(F_x^{\epsilon_n},\mathcal{S},W))- \cN(\phi_x^{\epsilon_n},\mathcal{S},W))\right) dx~|\ge \gamma.
	$$
	However, Lemma \ref{first approx} gives $\Bint_{ B(R)}\norm{F_x^{\epsilon_n}- \phi_x^{\epsilon_n}}^2_{C^1(B(W'))}dx\to 0$
	as $\epsilon_n\to 0$, thus there exists a subsequence $n_j$ such that with a rescaling $B(R)$ to a ball of radius $1$
	\begin{align}
	\label{4.4}
	\norm{F_{R_{n_j}x}^{\epsilon_{n_j}}- \phi^{\epsilon_{n_j}}_{R_{n_j}x}}^2_{C^1(B(W'))}\to 0
	\end{align}
	$x$-almost surely as $j\to\infty$. However, \eqref{4.4} together with by the continuity of $\cN$ (Lemma  \ref{PropConvE2}), the Faber-Krahn inequality and the Dominated Convergence Theorem, gives
	$$
	|\Bint_{B(1)^*} \cN(F_{R_{n_j}x}^{\epsilon_{n_j}},\mathcal{S},W))- \cN(\phi_{R_{n_j}x}^{\epsilon_{n_j}},\mathcal{S},W)) dx~|\to 0
	$$
	as $j\to \infty$, a contradiction. Finally, bearing in mind the Faber-Krahn inequality, we have the bound
	\begin{align*}
		\Bint_{B(R)}(\cN(F_x^\epsilon,\mathcal{S},W))- \cN(\phi_x^\epsilon,\mathcal{S},W)))dx=&\Bint_{B(R)^*} (\cN(F_x^\epsilon,\mathcal{S},W))- \cN(\phi_x^\epsilon,\mathcal{S},W))) dx+\\
		&+O\left(W'^m\left(1-\vol_R\left(\phi^\epsilon\in C^1_*(W)\right)\right)\right)
	\leq \varepsilon,
\end{align*}
where the inequality holds as long as $\epsilon<\epsilon_0$ given by \eqref{6.3.1} and $R,N$ large enough according to Lemma \ref{PropConvMeasure1} to ensure \eqref{ProbC*}. 

Hence, we make the following choices of the parameters in order to prove the proposition (recall that the parameters $K_0,N_0,R_0$ must be chosen to satisfy i) and ii)), let $K$ large enough, with $\delta\lesssim K^{-m+1}$ as in Lemma \ref{PropConvMeasure2.1}, such that $CW'^{2(s+l)+m}\times$ $\times \left(\delta K^{m-1} + W'^2K^{-2}\right)<\epsilon_0/2$ and such that $K>K_0(\varepsilon/2,W')$ accordingly to Lemma \ref{middle step}. Similarly, let $N>N_0(\varepsilon/2,W')$ with $N_0$ given in Lemma \ref{middle step}. We also take $N$ large enough according to Lemma \ref{PropConvMeasure1} such that	
\begin{equation}\label{ErrorProbPhi}
	O\left(W'^m\left(1-\vol_R\left(\phi^\epsilon\in C^1_*(W)\right)\right)\right)<\varepsilon/2
\end{equation}	
provided $R$ is large enough and the same for \eqref{5.3.1}. Finally, let $R$ large enough according to the two conditions mentioned below \eqref{ErrorProbPhi} and such that in the definition of $\epsilon$ in \eqref{epsilon} we have $O_N(R^{-\Lambda-3/2})<1$.
\end{proof}

\subsection{Concluding the proof of Theorem \ref{thm 5}}
\label{proof of Theorem m=2}
We are finally ready to prove Theorem \ref{thm 5}
\begin{proof}[Proof of Theorem \textnormal{\ref{thm 5}}]
	
Let $\mathcal{S} \subset H(m-1)$ and $T \in \mathcal{T}$ be given, and denote by $\mathcal{N}(f,\cdot,R)$ either $\mathcal{N}(f,\mathcal{S},R)$ or $\mathcal{N}(f,T,R)$. 	Thanks to Proposition \ref{semi-locality m>2} and the fact that the number of nodal domain with fixed topological class or tree type intersecting $B(W)$ is bounded by the total number of nodal domains intersecting $B(W)$, for $W>1$, we have 
\begin{align}
\label{SemilocalityProofMainTh}
	\frac{\mathcal{N}(f,\cdot,R)}{\vol B(R)}=  \frac{1}{\vol B(W)}\Bint_{ B(R)}\mathcal{N}(F_x,\cdot,R) dx
	+ O\left(W^{-1} \right) + O_{N,W}\left(R^{-\Lambda-3/2}\right).
\end{align}
		Now, by Theorem \ref{ThNSSWatomic}, with the same notation, we have 
		\begin{align}
		\frac{\mathbb{E}[ \mathcal{N}(F_{\mu},\cdot, W) ]}{\vol B(W)}= c_{NS}(\cdot, \mu) + a(W) \nonumber,
		\end{align}	
		where $a(W)\rightarrow 0$ as $W\to\infty$.
		By Proposition \ref{PropConvE2} applied with some $\varepsilon_0/2$, we have 
			\begin{align}
			\label{6.1.2}
			\left|\Bint_{B(R)}\cN(F_x, \cdot)dx-\bE[\mathcal{N}(F_{\mu}, \cdot, W)]~\right|<\varepsilon_0
			\end{align}
		if  $K\geq K_0$ with $K_0=K_0(\varepsilon_0/2,W)$, $N\geq N_0$ with $N_0=N_0(K,\varepsilon_0/2,W)$ and $R\geq R_0$ with $R_0=R_0(N,\varepsilon_0/2,W)$. 
		Hence,	putting \eqref{SemilocalityProofMainTh}, \eqref{6.1.2} together, we obtain 
		$$
		\left|\frac{\mathcal{N}(f,\cdot, R)}{\vol{B(R)}}-c_{NS}(\cdot, \mu)\right|\le\frac{ \varepsilon_0}{\vol B(W)} +a(W)+ 	O\left( \frac{1}{W} \right) + O_{N,W}\left(R^{-\Lambda-3/2}\right).
		$$
	Let us now pick some $\varepsilon>0$ and choose first a $W$ large enough (and fix it) so that
		$$
		O\left(W^{-1}\right)+a(W)<\varepsilon/3,
		$$
		then set $K\geq K_0$, $N\geq N_0$ and $R\geq R_0$ with $\varepsilon_0=\vol B(W)\varepsilon/3$ and $R$ large enough such that 
		\begin{equation*}
			O_{N,W}\left(\frac{1}{R^{\Lambda+3/2}}\right)<\varepsilon/3.
		\end{equation*}
\end{proof}

\section{Proof of Theorem \ref{thm 2}.}
\label{nodal volume}
\subsection{Uniform integrability of $\mathcal{	V}(F_x,W)$}
\label{proof of thm 2, I}
We first prove Proposition \ref{anti-concentration prop}:

\begin{proof}[Proof of Proposition \textnormal{\ref{anti-concentration prop}}]
By Lemma \ref{doubling f} for $r=W$, we have 
	\begin{align*}
			\Bint_{B(R)}\mathcal{V}(F_x,W)^{1+\alpha}dx\lesssim (W^{m})^{1+\alpha} + W^{(m-1)(1+\alpha)}\Bint_{B(R)} \mathfrak{N}_f(B(x,3W))^{1+\alpha} dx \nonumber \\
			= (W^{m})^{1+\alpha} + W^{(m-1)(1+\alpha)}\int_1^{CN^{1+\alpha}} \vol_R\left(\mathfrak{N}_f(B(x,3W))^{1+\alpha}>t\right) dt \nonumber 
	\end{align*}
 Changing variables and using Lemma \ref{anti-concentration v2} with $D=\alpha +2$ and $H=3W$,  we find that
 \begin{align}
 \int_1^{CN^{1+\alpha}} \vol_R\left(\mathfrak{N}_f(B(x,W))^{1+\alpha}>t\right) dt
 &= (1+\alpha) \int_1^{CN} t^{\alpha} \vol_R\left(\mathfrak{N}_f(B(x,3W))>t\right) dt \nonumber \\ &\lesssim W^{(1+\alpha)} +O_{N,W}(R^{-\Lambda-3/2}) \nonumber
 \end{align}
as required. 
\end{proof}

\subsection{ Continuity of $\mathcal{V}$.}
\label{proof of thm 2,II}
In this section we prove the following proposition
\begin{prop}\label{almost Theorem 3.4}
	Let $W>1$ be given and $f$ be as in \eqref{function}.	Then we have 
	$$
\lim_{N\to\infty}\limsup_{R\to \infty}	\left| \Bint_{ B(R)}{\mathcal{V}(F_x,W)}- {\mathbb{E}[\mathcal{V}(F_{\mu}, W)]}\right|=0 
	$$
\end{prop} 
The proof of Proposition \ref{almost Theorem 3.4} is similar to the proof of Proposition \ref{PropConvE2}. We just need the following lemma, which follows from Theorem \ref{ThThom}.
\begin{lem}
	\label{continouty nodal length}
	$\mathcal{V}(\cdot, B)$ is a continuous functional on $C^1_{*}(W)$. 
\end{lem}  
Note that now the condition of $\slashed{\Psi}>0$ is not needed here.                              
\begin{proof}
 For the sake of simplicity, we set $V=B(W+1)$ and $\chi$ as in the proof of Proposition \ref{PropContBdd} so we will only consider boundaryless manifolds. Let $h$ be an arbitrary function of $C^1_*(W)$. We apply Thom's Isotopy Theorem \ref{ThThom} to $h'\coloneqq h+\chi$,  which is identical to $h$ in $B(W)$. Let $\varepsilon>0$ and $\psi$ a local parametrization of a nodal component $L_{h'}$ of $h'$, which is a boundaryless manifold in $B(W+1)$. By Thom's Isotopy Theorem \ref{ThThom}, and with the same notation, $\Phi\circ\psi$ is a local parametrization of a nodal component of $g'$, $L_{g'}$ provided $\| h'- g'\|_{C^1(U)}<\delta$ and $\delta>0$ is small enough as in \eqref{eq:delta cond}. If $J$ denotes the Jacobian matrix and the local parametrization is $\psi:A\to \mathcal{U}$, we have 
\begin{equation*}
	\vol{(\mathcal{U}\cap L_{g'})}=\int_A \sqrt{\det ((J\psi(x))^t\cdot J\psi(x))}dx
\end{equation*}
and by the chain rule if $\Phi\circ\psi:A\to \mathcal{U}'$
\begin{equation*}
	\vol{(\mathcal{U}'\cap L_{g'})}=\int_A \left|\det (J\Phi)(\psi(x))\right| \sqrt{\det ((J\psi(x))^t\cdot J\psi(x))}dx.
\end{equation*}
But, by the standard series for the determinant:
\begin{equation*}
	\det (J\Phi(y))=\det (I+J\Phi(y)-I)=1+\textnormal{tr}(J\Phi(y)-I)+o\left(\norm{J\Phi(y)-I}\right)
\end{equation*}
where $I=J\textnormal{id}$ is the identity and
\begin{equation*}
	|\textnormal{tr}(J\Phi(y)-I)+o\left(\norm{J\Phi(y)-I}\right)|\lesssim \norm{\Phi-\textnormal{id} }_{C^1(\R^m)}<\varepsilon.
\end{equation*}
 Thus, for $L_{h}\subset \bigsqcup_{i} \mathcal{U}_i\subset B(W+1)$ we have $|\vol{(L_{h'})}-\vol{(L_{g'})}|\lesssim \varepsilon \vol{(L_{h'})}.$ Hence, we also have
\begin{equation*}
	|\vol{(L_{g'}\cap B(W))}-\vol{(\Phi^{-1}(L_{g'}\cap B(W)))}|\lesssim \varepsilon \vol{(L_{g'}\cap B(W))}\lesssim \varepsilon (1+\varepsilon) \vol{(L_{h'})}.
\end{equation*}
Finally, as 
$$
|\vol{(\Phi^{-1}(L_{g'}\cap B(W)))}-\vol{(L_{h'}\cap B(W))}|\le \vol{(L_{h'}\cap B(W+\varepsilon)\backslash B(W))}
$$
the continuity of $\mathcal{V}$ at $h$ follows immediately as the are finitely many components $L_{h'}$ (because $h\in C^1_{*}(W)$), by the fact that $\norm{h-g}=\norm{h'-g'}$ for $g'=g+\chi$ and $L_{g'}\cap B(W)=L_{g}\cap B(W)$ (the same for $L_h$).
\end{proof}
We are now ready to prove Proposition  \ref{almost Theorem 3.4}.
\begin{proof}[Proof of Proposition \textnormal{\ref{almost Theorem 3.4}}]
	The proof is similar to the proof of Proposition \ref{PropConvE2} so we omit some details. By Claim \ref{open} and Lemma \ref{Bulinskaya's Lemma} we have
	\begin{align}\label{7.2.2}
	 \bP\left(F_\mu\in C^1_*(W)\right)= 1.
	\end{align} 
Now,  Lemma \ref{continouty nodal length} together with \eqref{7.2.2}, Theorem \ref{thm 1} (applied with $W+1$) and the Continuous Mapping Theorem  imply that 
\begin{align}
	\mathcal{	V}(F_x, W)\overset{d}{\longrightarrow} \mathcal{	V}(F_{\mu}, W)  \label{7.2.3}
\end{align}
where the convergence is in distribution as $R,N\rightarrow \infty$ according to that theorem. Hence, Proposition \ref{almost Theorem 3.4} follows from \eqref{7.2.3}, Proposition \ref{anti-concentration prop} and Lemma \ref{th:DCTh}. 
\end{proof}
\subsection{Concluding the proof of Theorem \ref{thm 2}}
Before concluding the proof of Theorem \ref{thm 2}, we need to following direct application of the Kac-Rice formula: 
\begin{lem}
	\label{Kac-Rice}
	Suppose that $F_{\nu}$ is a centred, stationary Gaussian field defined on $\mathbb{R}^m$ such that $(F,\nabla F)$ is non-degenerate and $\nu$ is supported on $\S^{m-1}$ then there exist some constant $c=c(\nu)$ such that
	\begin{align}
	\mathbb{E}[\mathcal{V}(F_{\nu},W)]= c(\nu)\vol B(W) \nonumber
	\end{align}
\end{lem}
\begin{proof}
	For brevity let us write $F=F_{\nu}$. By \cite[Theorem 6.3]{AT}, since $F$ is non-degenerate \eqref{non-degeneracy} and almost surely analytic, we have 
	\begin{align} \label{1.1.1}
	\mathbb{E}[\mathcal{V}(F_{\nu},W)]= \int_{B(W)}\mathbb{E}[|\nabla F(y)| | F(y)]v_{F(y)}(0)dy
	\end{align}
	where $v_{F(y)}(0)$ is the density of $F(y)$ at zero. By stationarity, the integrand in \eqref{1.1.1} is independent of $y$ and the Lemma follows. 
\end{proof}
\begin{rem}
	It is possible to explicitly derive an expression for $	\mathbb{E}[\mathcal{V}(F_{\nu},W)]= \int_{B(W)}\mathbb{E}[|\nabla F(y)| | F(y)]$ in terms of the covariance of $F_{\nu}$ and its derivatives, following computations similar to those of \cite{KKW,EPR21}. However, since the calculations are quite long, we decided not to include them in the article. 
\end{rem}
And lastly the following lemma: 

\begin{lem}\label{locality nodal volume} Let $h:\R^m\to \R$ and $0<r<R$. Then we have
	\begin{equation}
		\cV(h,R-r)\le \int_{B(R)}\frac{\cV(h;x,r)}{\vol{B(r)}} dx \le \cV(h,R+r)
	\end{equation}
\end{lem}
\begin{proof}
	By the definition of $\cV$ and Fubini, we have 
	\begin{align} 
	\int_{B(R)}{\cV(h;x,r)} dx&=	\int_{B(R)}\int_{B(R+r)}\mathbbm{1}_{B(x,r)}(y)\mathbbm{1}_{h^{-1}(0)}(y) d\cH(y)dx. \nonumber \\
	&=\int_{B(R+r)} \mathbbm{1}_{h^{-1}(0)}(y)\vol\left(B(y,r)\cap B(R)\right) d\cH(y), \nonumber
	\end{align}
	so the lemma follows from
	$
	\mathbbm{1}_{B(R-r)}\le\dfrac{\vol\left(B(\cdot,r)\cap B(R)\right)}{\vol{B(r)}}\le \mathbbm{1}_{B(R+r)}.
	$
\end{proof}

We are finally ready to prove Theorem \ref{thm 2}. 
\begin{proof}[Proof of Theorem \textnormal{\ref{thm 2}}] The proof follows closely the proof of Theorem \ref{thm 3}, so we omit some details. Let $\varepsilon>0$ be given, then applying	 Lemma \ref{locality nodal volume} with $r=W$ and dividing by $\vol{B(R)}$, we have
	$$
	\frac{1}{\vol B(R)}\int_{B(R-W)}\frac{\cV(F_x,W)}{\vol{B(W)}} dx\le \frac{\cV(f,R)}{\vol{B(R)}}\le\frac{1}{\vol B(R)} \ \int_{B(R+W)}\frac{\cV(F_x,W)}{\vol{B(W)}} dx. 
	$$
For any $\alpha>0$, Proposition \ref{anti-concentration prop} gives
	\begin{equation*}
		\frac{1}{\vol B(R)}\int_{B(R+W)\backslash B(R)}\frac{\cV(F_x,W)}{\vol{B(W)}}=O_{N,W,\alpha}(R^{-\gamma}),
	\end{equation*}
for some $\gamma>0$. Therefore, as in the proof of Proposition \ref{semi-locality m>2},
	\begin{equation}\label{eq:semiloc vol}
		\frac{\cV(f,R)}{\vol{B(R)}}=\Bint_{B(R)}\frac{\cV(F_x,W)}{\vol{B(W)}}dx+O_{N,W,\alpha}(R^{-\gamma}).
	\end{equation}
Finally, Proposition \ref{almost Theorem 3.4} gives, for every $\varepsilon>0$,  
	$$
	\frac{\mathcal{V}(f,R)}{\vol B(R)}=  \frac{1}{\vol B(W)}\mathbb{E}[\mathcal{V}(F_{\mu},W)] +O( \varepsilon) 
	$$
	for all $N$ and $R$ sufficiently large. The theorem now follows from Lemma \ref{Kac-Rice}. 
\end{proof}

\section{Final comments.}\label{comments}

\subsection{Exact Nazarov-Sodin constant for limiting function?}

As we have seen, Theorem \ref{thm 3} says that there are deterministic functions with the growth rate for the nodal domains count arbitrarily close, increasing $N$, to the Nazarov-Sodin constant. One may wonder whether this constant is attained if $N$ goes to infinity and the functions $\{f_N\}$ as in \eqref{function} (or a rescaling of it) converges, in some appropriate space of functions, to some function $f$. Our argument does not apply outright in the limit $N\rightarrow \infty$ and then $R\rightarrow \infty$. That is, Theorem \ref{thm 3} gives
\begin{equation}
\lim_{N\to\infty}\limsup_{R\to \infty}\left|\frac{\cN(f_N,R)}{\vol B(R)}-c_{NS}\right|=0. \label{9.2}
\end{equation}
However, given a sequence of functions $\tilde{f}_N\coloneqq C_N f_N$ with $C_N>0$ (so $\cN(f_N,R)=\cN(\tilde{f}_N,R)$) such that $\tilde{f}_N\to f$, one could hope that the following holds:
\begin{equation} \label{9.1}
	\lim_{R\to\infty}\lim_{N\to \infty}\left|\frac{\cN(f_N,R)}{\vol B(R)}-c_{NS}\right|=\lim_{R\to\infty}\left|\frac{\cN(f,R)}{\vol B(R)}-c_{NS}\right|=0. 
\end{equation}
 In this section we show that this is not true in general, that is, we give examples of sequences of functions such that $\tilde{f}_N\to f$ and \eqref{9.2} hold but \eqref{9.1} does not, in fact, the growth rate is much smaller. 

	Let us consider the functions
	\begin{equation} \nonumber
	f_N\coloneqq \frac{1}{2N^{1/2}}\sum_{|n|\leq N} e(\langle r_n,  x\rangle),
	\end{equation}
	assume that $\mu$ is the Lebesgue measure on the sphere and define $\tilde{f}_N= N^{-1/2}f_N$. Then,  we have 
	$$
	\tilde{f}_N(x)=\frac{1}{{2N}}\sum_{|n|\leq N} e(\langle r_n,  x\rangle)\overset{N\rightarrow \infty}{\longrightarrow}\int_{\S^{n-1}}e(\langle \omega,  x\rangle) d\sigma(\omega)=:f =  C_m\frac{ J_{\Lambda}\left(2\pi|x|\right)}{|x|^{\Lambda}} $$
where $\Lambda=(m-2)/2$ and the convergence is uniform on compact sets, see \eqref{A.1},  with respect to $x\in \mathbb{R}^m$ and also holds after differentiating any finite number of times.  Thus by Theorem \ref{thm 1}, \eqref{9.2} holds, but 
	\begin{align}
\cN(f,R)=cR(1+o_{R\rightarrow \infty}(1)) \label{9.3}
	\end{align}
	for some known $c>0$, thus \eqref{9.1} does not hold. 	
	We also observe that, using \cite[Theorem 3.1]{EPR19}, it is possible to make more general choices of $a_n$, for example $a_n=\phi(r_n)$ for some sufficiently smooth function $\phi$. With this choice,  either the number of nodal domains of $f$ grows as in \eqref{9.3} or it could even be bounded (for large enough $R$ the nodal set is a non-compact nodal component consisting on layers and an ``helicoid'' connecting them).
	
	 In order to illustrate this change, we show in Figure \ref{FigNod1} below the nodal set for the function
	$$
	g_N(x,y)\coloneqq \frac{1}{N}\sum _{n=1}^N \cos (x \cos \theta_n+ y \sin \theta_n)
	$$
	for different $N$ and $\theta_n$ uniformly distributed over the sphere. As $N$ increases the connected components of the nodal sets near the origin tend to merge and they are close and diffeomorphic to the ones of $J_0$.
	
	Finally, we mention that a similar phenomenon happens for eigenfunctions on the two dimension torus $\mathbb{T}^2=\mathbb{R}^2/\mathbb{Z}^2$. These can be written as 
	\begin{equation*}
	f_{\mathbb{T}^2}(x)= \sum\limits_{\xi\in\mathcal{E}}a_\xi e(\langle \xi,x\rangle),
	\end{equation*}
	where $\mathcal{E}=\mathcal{E}_{E}=\{\xi\in\mathbb{Z}^2:|\xi|^2=E\}$ and $a_{\xi}$ are complex coefficients.  Under some arithmetic conditions and some constrains on the coefficients, in \cite{BU,BW} it is showed that there  are deterministic realizations of the RWM, that is
	\begin{equation}\label{eq:nodal TT}
	\cN(f)= c_{NS}(\mu)\cdot E(1+o(1)).
	\end{equation}	
However, since the points $\xi/\sqrt{E}\in \S^1$ ``generically'' become equidistributed on $\S^1$ \cite{EH99}, the function 
$$\tilde{f}= \frac{1}{|\mathcal{E}|} \sum\limits_{\xi\in\mathcal{E}} e(\langle \xi,x\rangle), $$
once rescaled, presents the same limiting behaviour described by \eqref{9.3}. In particular, for $f_{1/\sqrt{E}}\coloneqq f\left(E^{-1/2}\right)$, considering the periodicity, \eqref{eq:nodal TT} and $R$ large enough,
\begin{equation*}
	\cN(f_{1/\sqrt{E}},R)\sim \vol B(R)\cdot c_{NS}(\mu)(1+o(E^0)),
\end{equation*}
still an approximated constant.
\begin{figure}[h!] 
	\begin{subfigure}[b]{0.40\textwidth}
		\includegraphics[width=\textwidth,height=1.05\textwidth]{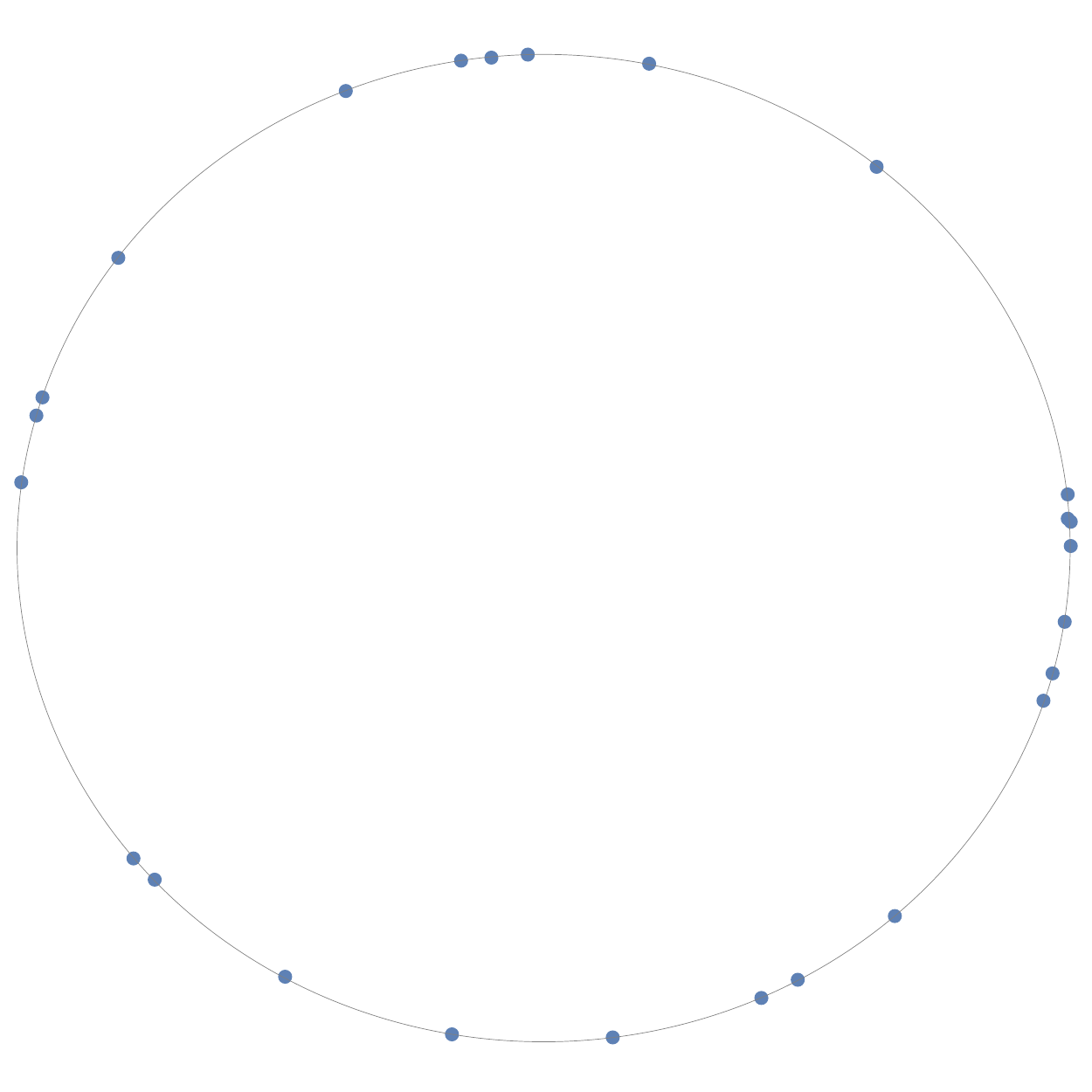}
		\caption{25 points.}
	\end{subfigure}
	\begin{subfigure}[b]{0.40\textwidth}
		\includegraphics[width=\textwidth,height=1.05\textwidth]{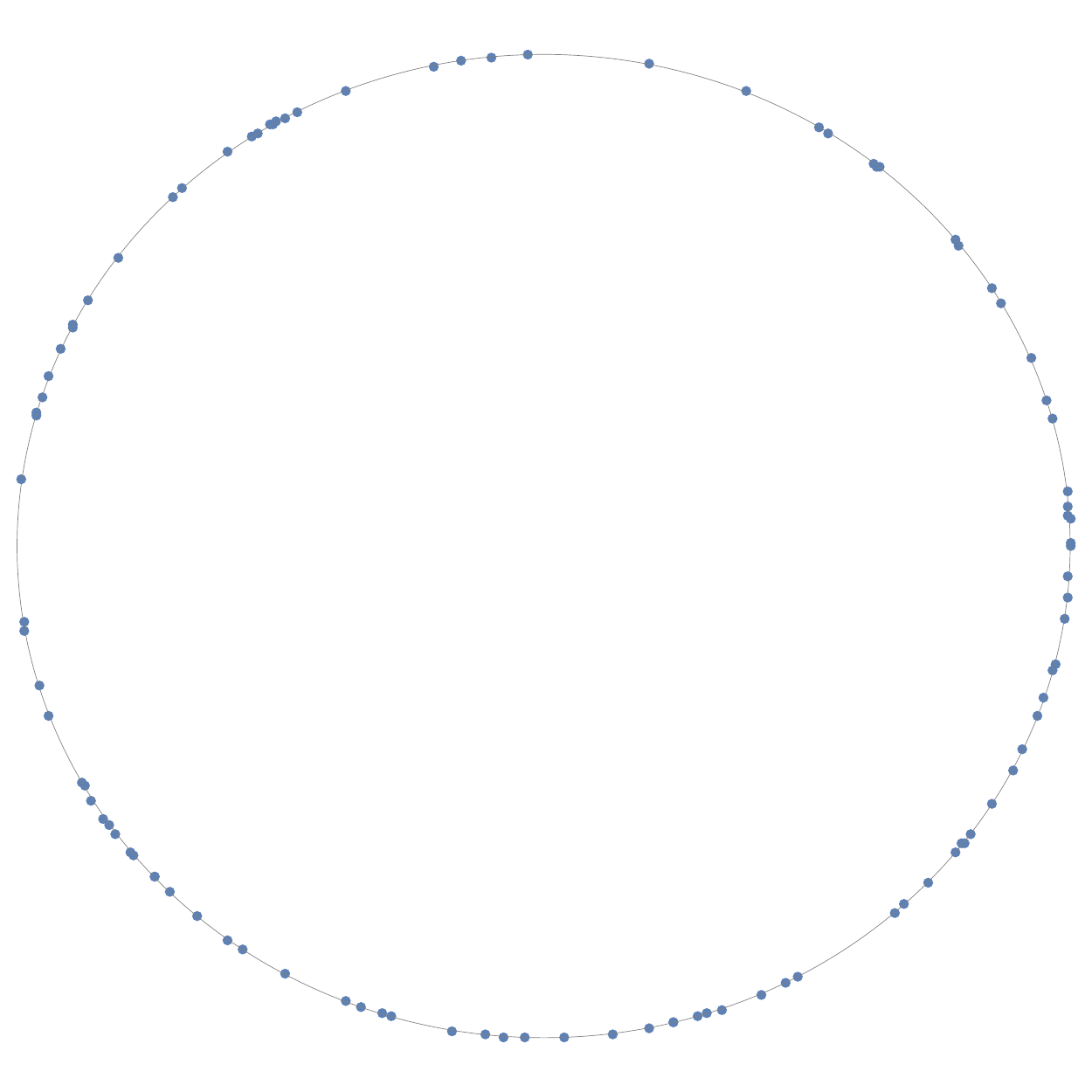}
		\caption{100 points.}
	\end{subfigure}  
	\begin{subfigure}[b]{0.40\textwidth}
		\includegraphics[width=\textwidth]{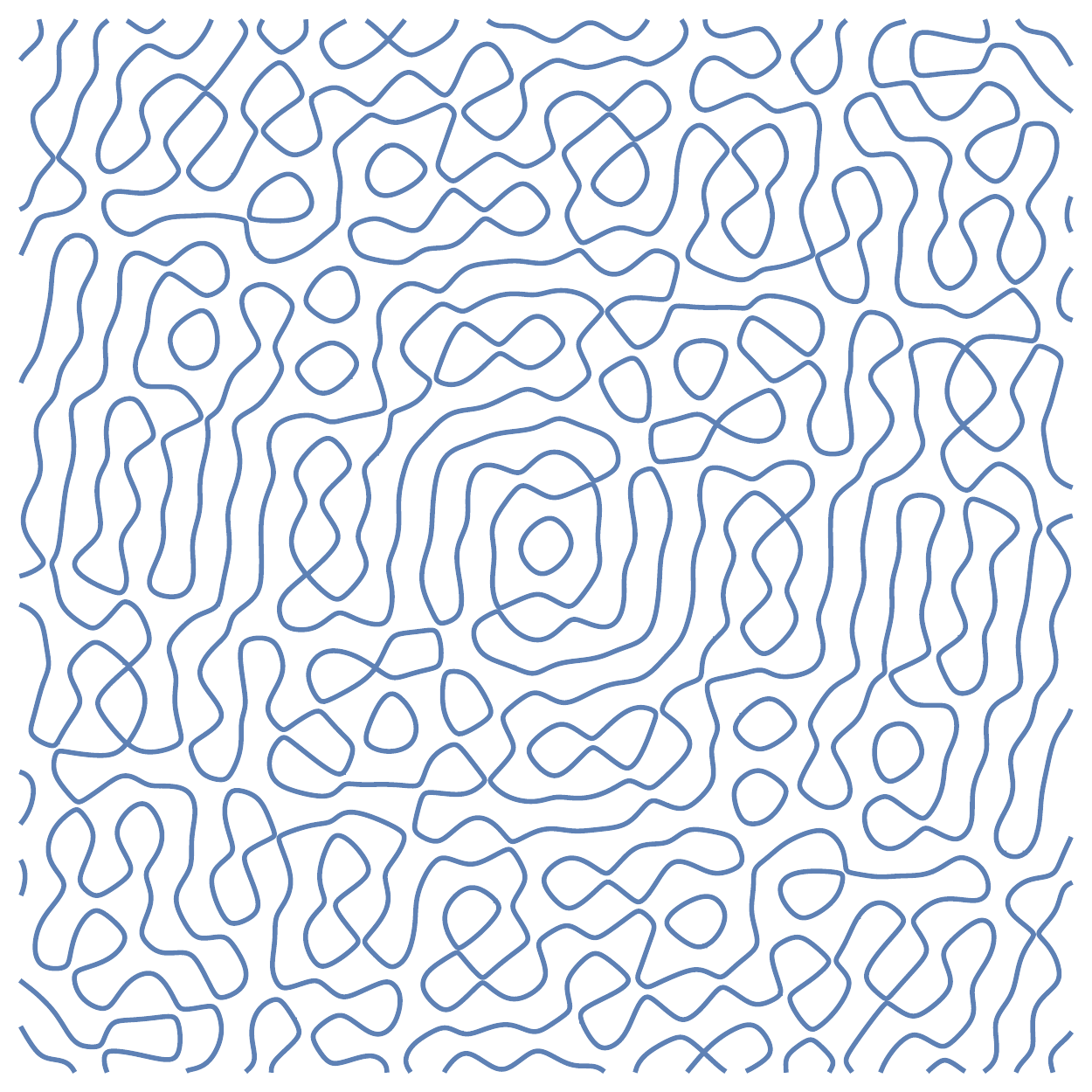}
		\caption{Nodal set for the function with 25 points.}
	\end{subfigure} 
	\begin{subfigure}[b]{0.40\textwidth}
		\includegraphics[width=\textwidth]{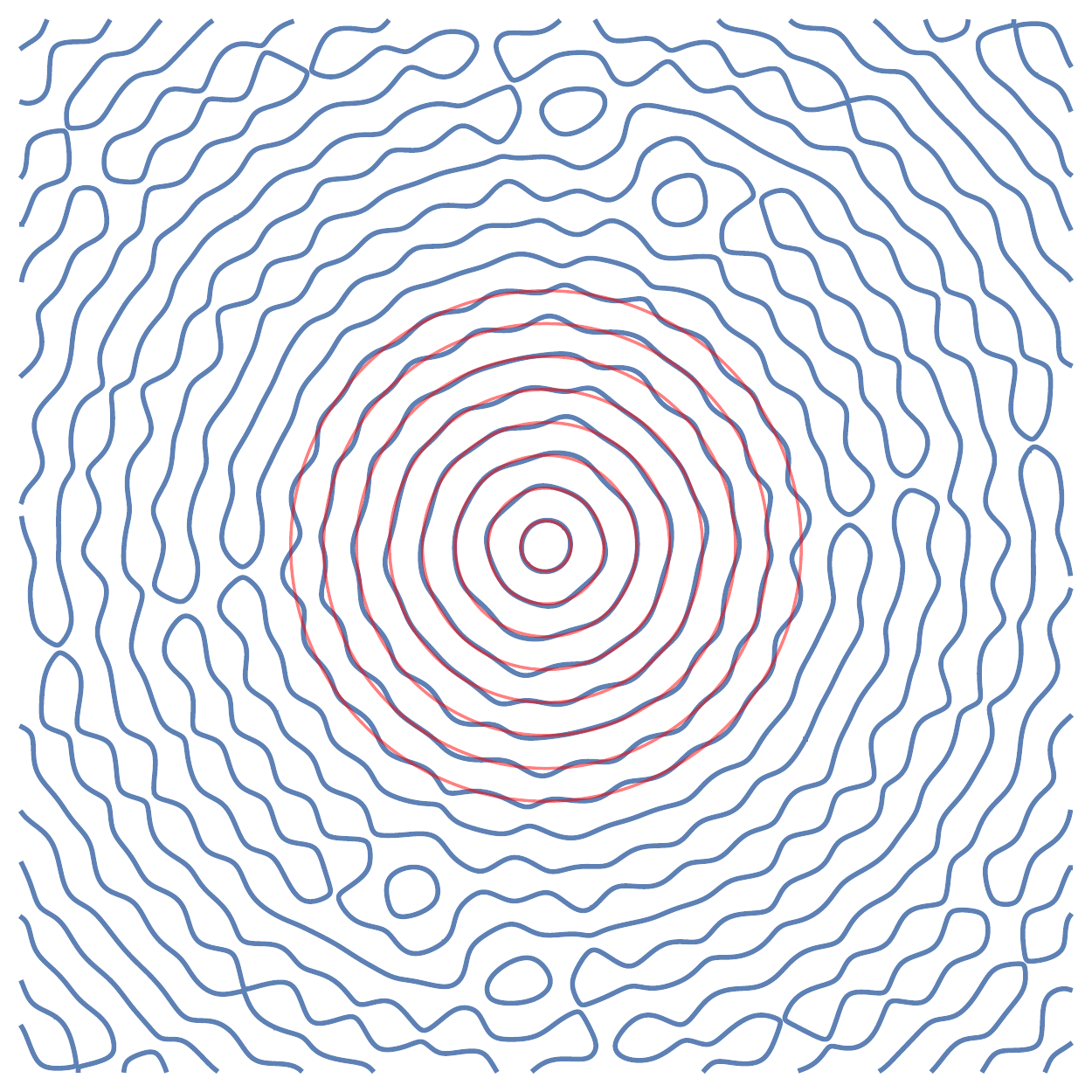}
		\caption{Nodal set for the function with 100 points. In red for $J_0(|x|)$.}
	\end{subfigure}         
	\caption{Nodal set for the function $g_N$ for different $N$ and the points on the sphere $\S^1$. We have represented in red the first eight connected components of $J_0(|\cdot|)^{-1}(0).$} \label{FigNod1}
\end{figure}

\subsection{On a question of Kulberg and Wigman}\label{sec:KW}
The methods we have developed allow us to solve a question raised by Kulberg and Wigman, \cite[Section 2.1]{KW2}, on the continuity of
$$
\mu\mapsto \mathbb{E}[\mathcal{N}(F_{\mu},R)]
$$
on the plane. Indeed, we can strengthen Lemma \ref{middle step} to give a more general result as in Proposition \ref{prop:cont conj}. This solves their question on any dimension, not only $m=2$.
\begin{prop}\label{prop:cont conj}
	Let $\mu_n$ be a sequence of measures on the sphere $\mathbb{S}^{m-1}$ not supported on a hyperplane converging weakly to $\mu$ as $n\to\infty$. Then
	\begin{equation*}
		\mathbb{E}[\mathcal{N}(F_{\mu_{n}},R)]\to\mathbb{E}[\mathcal{N}(F_{\mu},R)]
	\end{equation*}
	for any $R>0$.
\end{prop}
\begin{proof}
	Thanks to Lemma \ref{Bulinskaya's Lemma}, $F_{\mu_{n}},F_{\mu}\in C_{*}^1(B(W))$ almost surely. From Lemma \ref{LemConvGaussianFields} we can stablish $d_P(F_{\mu_n},F_\mu)\to 0$. Thus the proposition follows directly Portmanteau Theorem,  as  
	$\mathcal{N}(F_{\mu_n},R)=O(R^m)$ by the Faber-Krahn inequality, uniformly for all $n$, and it is continuous by Proposition \ref{PropContBdd}. 
\end{proof}
\begin{rem}
	The same holds for topological classes and trees with an analogous, \textit{mutatis mutandis}, proof.
\end{rem}
With our techniques we can also extend the discrepancy functional that was introduced in \cite[Proposition 1.2]{KW2} for any dimension, topologies and nesting trees. More precisely, we have:
\begin{prop} The discrepancy functional exists, that is,
	\begin{equation}\label{DiscFunct}
	\lim_{R\to\infty}\bE\left|\frac{\mathcal{N}(F_\mu,\cdot,R)}{\vol S(R)}-c(\cdot,\mu)\right|
	\end{equation}	  
exists and it is finite. 
\end{prop}	
\begin{proof}
	Following the proof of Theorem \ref{ThNSSWatomic},
	$$
	\lim_{R\to\infty}\bE\left|\frac{\mathcal{N}(F_\mu,\cdot,R)}{\vol S(R)}-c(\cdot,\mu)\right|=\bE\left|c(G,\cdot,\mu)-c(\cdot,\mu)\right|,
	$$
	as the convergence is in $L^1$. 	
\end{proof}
	From the proof we see that the discrepancy functional, \eqref{DiscFunct}, is zero if and only if $c(F_\mu,\cdot,R)$ is a.s. a constant, that is, the limit of the nodal counts is non-random. This is true, in particular, if the field is ergodic. This functional measures how far we are from the ergodic situation of $\lim_{R\to\infty} \frac{\mathcal{N}(G,\cdot,R)}{\vol S(R)}$ being, a.s., a constant.
\section*{Acknowledgement.}
The authors would like to thank Alberto Enciso, Daniel Peralta-Salas, Mikhail Sodin and Igor Wigman for many useful discussions, Alejandro Rivera and Maxime Ingremeau for stimulating conversations and Alon Nishry for valuable comments on the first draft of the article.  The first author would like to thank professor Igor Wigman for supporting his visit to the King's College in London. The second author would like to thank Alberto Enciso, Daniel Peralta-Salas for supporting his visit to the ICMAT in Madrid and  G. Baldi for suggesting the construction in Example \ref{example2}. A.R. is
supported by the grant MTM-2016-76702-P of the Spanish Ministry of Science and in part by the ICMAT--Severo Ochoa grant
SEV-2015-0554. A.R. is also a postgraduate fellow of the City Council of Madrid at the Residencia de Estudiantes (2020--2022). A.S. was supported by the Engineering and Physical Sciences Research Council [EP/L015234/1], the EPSRC Centre for Doctoral Training in Geometry and Number Theory (The London School of Geometry and Number Theory), University College London.

\appendix
\section{Gaussian fields lemma.}\label{AppendixGaussianFields}

\label{descarding phi}
Adapting \cite[Lemma 7.2]{EPRBeltrami} for our situation we can show:
\begin{lem}
	\label{LemConvGaussianFieldsApp}
	Let $s\geq 0$ and $W\geq 1$, be given. Moreover, suppose that $\{\mu_n\}_{n\in \mathbb{N}}$ be a sequence of probability measures on $\mathbb{S}^{m-1}$ such that $\mu_n$ weak$^{\star}$ converges  to $\mu$. Then,
	\begin{align}
d_{P}(F_{\mu_n},F_{\mu})\longrightarrow 0 \text{ as }~n \rightarrow \infty \nonumber
	\end{align}
	where the convergence is with respect to the $C^s(B(W))$ topology. 
\end{lem}

\begin{proof}
	 Since $\mu_n$  weak$^{\star}$-converges  to $\mu$ and the exponential is a bounded continuous function, by the discussion in Section \ref{Gaussian random fields} and Portmanteau Theorem, we have for any $x,y\in B(W)$
	 \begin{align}\label{A.1}
	 \begin{split}
	 K_n(x,y)\coloneqq\mathbb{E}[F_{\mu_n}(x)F_{\mu_n}(y)]=\int e(\langle x-y,\lambda \rangle)d\mu_n(\lambda) \to\\
	 \to \int e(\langle x-y,\lambda \rangle)d\mu(\lambda)=
	 \mathbb{E}[F_{\mu}(x)F_{\mu}(y)]=K(x,y)
	 \end{split}	 
	 \end{align}	
	 Since, $\mu_n$ and $\mu$ are compactly supported, we may differentiate under the integral sign in \eqref{A.1}, it follows that \eqref{A.1} holds after taking derivatives.	 
	Thanks to the Cauchy-Schwarz inequality and the fact that $\mu_N$ is a probability measure, we have 
	$$
	|K_n(x,y)-K_n(x',y')|\le C\norm{x-x'+y-y'}
	$$
	as $|e^{ix}-e^{iy}|\le |x-y|$ for any $x,y\in\R.$ Therefore $K_n$ is equicontinuous, $|K_n(x,y)|\leq 1$ and it converges pointwise; thus, the Arzel\`a-Ascoli Theorem  implies that the convergence in \eqref{A.1} is uniform for all $x,y \in B(W)$, together with its derivatives. Now, for any integer~$t\geq0$, the mean of the $H^t$-norm of~$F_{\mu_n}$ is uniformly bounded:
	\begin{align*}
	\bE\|F_{\mu_n}\|_{H^t(B(W))}^2=\sum_{|\alpha|\leq t}\bE\int_{B(W)}
	|D^\alpha F_{\mu_n}(x)|^2\, dx
	=\sum_{|\alpha|\leq t}\int_{B(W)}  D_x^\alpha D_y^\alpha
	K_{n}(x,y)\big|_{y=x}\, dx \\
	\xrightarrow[n\to\infty]{\phantom{n}}
	\sum_{|\alpha|\leq t}\int_{B(W)} D_x^\alpha
	D_y^\alpha K(x,y)\big|_{y=x}\, dx<M_{t,W}\,,
	\end{align*}
	where, in the last line, we have used \eqref{A.1}. As the constant~$M_{t,W}$ is independent of~$N$, Sobolev's
	inequality ensures that one can now take any sufficiently large $t$ to
	conclude that
	\[
	\sup_n\bE \|F_{\mu_n}\|_{C^{s+1}(B(W))}^2\leq C\sup_n\bE \|F_{\mu_n}\|^2_{H^t(B(W))}<M
	\]
	for some constant~$M$ that only depends on~$W,t$. Thus, for any~$\varepsilon>0$ and any sufficiently large $n$, we have 
	\[
	\nu_{n}^W\big(\big\{ F_{\mu_n}\in C^{s+1}(B(W)): \|F_{\mu_n}\|_{C^{s+1}(B(W))}^2>M/\varepsilon\big\}\big)<\varepsilon\,,
	\]
	where $\nu_{n}^W$ is the measure on the space of $C^s(B(W))$ functions corresponding to the random field $F_{\mu_n}$. Accordingly, the sequence of probability measures $\nu^{W}_n$ is tight. Indeed, by Lagrange and Arzelà-Ascoli the closure of the set
	$$
	\big\{ F_{\mu_n}\in C^{s+1}(B(W)): \|F_{\mu_n}\|_{C^{s+1}(B(W))}^2\le M/\varepsilon\big\}
	$$
	is precompact with the $C^s$ topology, so we can conclude by the very definition of tightness, see Section \ref{SectWeakConv}.
	\end{proof}

\section{Upper bound on $\mathcal{NI}$.}
\label{copying}
In this section are going to prove Lemma \ref{abs bound} following \cite[Section 3.2]{CLMM20} (that is, the following argument is due to F. Nazarov and we claim no originality). We  will need the following (rescaled) result \cite[Lemma 2.5]{CLMM20}:
\begin{lem}
	\label{Lemma 2.5 CLMM20}
	Let $r>1$ and $\rho>0$, let $f$ be a Laplace eigenfunction with eigenvalue $4\pi^2r^2$ on $B(\rho)$. Suppose that an open set $\Omega\subset B(\rho)$ is $c_0$-narrow (on scale $1/r$) and $f=0$ on $\partial \Omega \cap B(\rho)$. Then, for every $\varepsilon>0$, sufficiently small depending on $c_0$ and $m$,  if  
	$$ \frac{|\Omega|}{|B_{\rho}|}\leq \varepsilon^{m-1},$$ then 
	$$\sup_{ \Omega \cap B(\rho/2)} |f|\leq  e^{-c/\varepsilon}\sup_{\Omega \cap B(\rho)}|f| $$
\end{lem}                                                                                                                                                                                                                                                                                                    
We are finally ready to prove Lemma \ref{bound intersections}
\begin{proof}[Proof of Lemma \textnormal{\ref{bound intersections}}]
	First, we rescale $f$ to $f_r$ so that $\mathcal{	NI}(f,x,r)= \mathcal{	NI}(f_r,x,1)$ and  we may assume that every nodal domain is $c_0$-narrow. Let  $Z= \mathcal{	NI}(f_r,x, 1)$, $\Omega_i$ be the elements of $\mathcal{	NI}(f_r,x, 1)$,  $D=\mathfrak{N}_{f_{2r}}(B(x,1)) +r= \mathfrak{N}_f(B(x,2r))+r$ and finally let $B(x)=B(x,1)$. Suppose that $Z> 2^{m+2} c_1^{-1}\cdot  D^{m-1}$ for some constant $c_1=c_1(m)$ to be chosen later, we are going to derive a contradiction.
	
	 Let $x_i \in \Omega_i \cap B(x, 1)$ and define 
$$ S(p)=\left| \{\Omega_i: |\Omega_i \cap B(x_i,2^{-j})|\leq c_1^{-1}D^{-m+1} |B(x_i,2^{-j})| \hspace{1mm} \text{for} \hspace{1mm} j\in\{0,...,p\}\}  \right| .$$
First we are going to show that $S(0)\neq \emptyset$. Indeed, since 
$$\sum_i | \Omega_i \cap B(x_i,1)| \leq |B(2)|,$$ 
for at least $(3/4)Z$ nodal domains, we have 
$$ \frac{| \Omega_i \cap B(x_i,1)| }{2^m|B(1)|}\leq \frac{2^2}{Z},$$
thus $S(0)\geq (3/4)Z$.

Now, we claim the following:
\begin{claim}
	\label{claim Nazarov}
Let $p\geq 1$, then there are at most $Z 4^{-p-2}$ nodal domains $\Omega_i \in S(p)\backslash S(p+1).$
\end{claim}
\begin{proof}
 From now on, fix some $p\geq 1$ and assume that $\Omega_i\in S(p)$, we wish to apply Lemma \ref{Lemma 2.5 CLMM20} with $\varepsilon=c_1^{\frac{1}{m-1}}D^{-1}$, therefore we assume $c_1$ is sufficiently small in terms of $m$ and $c_0$.  Hence, we may apply Lemma \ref{Lemma 2.5 CLMM20}  to $B(x_i,2^{-j})$ for $j=0,...,p$, to see that 
 \begin{align}
 \label{AC.1}
\frac{\sup_{\Omega_i\cap B(x_i, 2^{-p})}|f_r|}{\sup_{B(x_i,2)}|f_r|}\leq (2^{-p-1})^{c_2D}
 \end{align} 
for some $c_2=c_2(m)>0$ and for all $i$. On the other hand, Lemma \ref{Remez for f} applied  to $f_r$, with $B(2)$ and $E=\cup_{\Omega_i\in S_p} \Omega_i \cap B(x_i,2^{-p-1})$, in light of the fact that nodal domains are disjoint, gives 
\begin{align}
	\left(\frac{\sum_i |\Omega_i \cap B(x_i,2^{-p-1})|}{|B(x_i,2)|}\right)^{CD}\lesssim \frac{\sup_{E} |f_r|}{\sup_B |f_r|} , \nonumber
\end{align}
which, together with \eqref{AC.1}, implies 
\begin{align}
	\label{AC.2}
		\left(\frac{\sum_i |\Omega_i \cap B(x_i,2^{-p-1})|}{B(x_i,2)}\right)^{CD}\lesssim (2^{-p-1})^{c_2D}
\end{align}
Rescaling we have  $B(x_i,2)= 2^{m(p+2)} B(x_i,2^{-p-1})$, thus \eqref{AC.2} can be rewritten as 
\begin{align}
	\left(\frac{\sum_i |\Omega_i \cap B(x_i,2^{-p-1})|}{B(x_i,2^{-p-1})}\right)^{CD}\lesssim (2^{-p-1})^{c_2D- mCD},
\end{align}
which, taking $c_2$ sufficiently large, that is, $c_1$ small depending on $m$, implies that 
$$\frac{\sum_i |\Omega_i \cap B(x_i,2^{-p-1})|}{B(x_i,2^{-p-1})} \leq 4^{-p-2}. $$
Hence there are at most $Z \cdot 4^{-p-2}$ nodal domains satisfying 
 \begin{align}
	\frac{|\Omega_i \cap B(x_i,2^{-p-1})|}{|B(x_i,2^{-p-1})|}> c_1 D^{-m+1},  \nonumber
\end{align}
and the claim follows.  
\end{proof}
Using the claim and the fact that $S(p+1)\subset S(p)$, we see that, for each $p\geq 0$ $|S(p)|\geq Z-Z\sum_p 4^{-p}\geq Z/2$. However, since the number of nodal domains is finite and 
\begin{align}
	\nonumber
	&\frac{|\Omega \cap B(x,\rho)|}{|B(x,\rho)|}= 1 \text{ for }\rho \text{ small enough, }
\end{align} 
we have that $S(p)$ is empty for $p$ sufficiently large, a contradiction. 
\end{proof}

\bibliographystyle{siam}
{\bibliography{DerandomBib}}

\end{document}